\documentclass[a4paper, 12 pt, reqno, oneside]{amsart}
\usepackage{amsfonts, amssymb, amsmath, eucal, amsthm, stmaryrd, verbatim}
\usepackage{graphicx}
\usepackage{bbm}
\usepackage{color}  % \usepackage[usenames]{color} 
\usepackage[hidelinks]{hyperref}

\usepackage[marginparwidth=2.8cm, vmargin=3.2cm, hmargin=3.4cm]{geometry}

\newtheorem{lemma}{Lemma}
\newtheorem{theorem}[lemma]{Theorem}
\newtheorem*{theorem*}{Theorem}
\newtheorem{proposition}[lemma]{Proposition}
\newtheorem{corollary}[lemma]{Corollary}

\theoremstyle{definition}
\newtheorem{definition}[lemma]{Definition}

\newtheorem{example}[lemma]{Example}

\newtheorem{remark}[lemma]{Remark}
\newtheorem{question}[lemma]{Question}

\numberwithin{lemma}{section}

\newcommand{\Gr}{\mathrm{Gr}}
\newcommand{\persmod}{\mathbf{PersMod}}
\newcommand{\grmod}{\mathbf{GrMod}}

\newcommand{\R}{\mathbb{R}}
\newcommand{\Z}{\mathbb{Z}}
\newcommand{\N}{\mathbb{N}}
\newcommand{\calR}{\mathcal{R}}
\newcommand{\calCh}{\mathcal{\check{C}}}
\DeclareMathOperator{\diam}{diam}

\newcommand{\kk}{\mathbbm{k}}
\DeclareMathOperator{\rank}{\mathrm{rank}}

\newcommand{\Top}{\mathbf{Top}}
\newcommand{\Vect}{\mathbf{Vect}}
\newcommand{\GrVect}{\mathbf{GrVect}}

\newcommand{\id}{\operatorname{id}}
\newcommand{\ind}{\operatorname{ind}}
\DeclareMathOperator{\tor}{Tor}
\DeclareMathOperator{\colim}{colim}
\DeclareMathOperator{\MH}{MH}
\DeclareMathOperator{\MC}{MC}
\DeclareMathOperator{\BMH}{BMH}
\DeclareMathOperator{\BMC}{BMC}

\usepackage{tikz}
\usepackage{tkz-berge}
\usepackage{ctable}
\newcommand{\otoprule}{\midrule[\heavyrulewidth]}

\title{Persistent Magnitude}
\author{Dejan Govc}
\address{Institute of Mathematics, University of Aberdeen, Aberdeen,
United Kingdom AB24 3UE}
\email{dejan.govc@abdn.ac.uk}
\thanks{DG was supported by EPSRC grant EP/P025072/1}

\author{Richard Hepworth}
\address{Institute of Mathematics, University of Aberdeen, Aberdeen,
United Kingdom AB24 3UE}
\email{r.hepworth@abdn.ac.uk}


\begin{document}

\begin{abstract}
In this paper we introduce the \emph{persistent magnitude}, a new numerical invariant of (sufficiently nice) graded persistence modules.  It is a weighted and signed count of the bars of the persistence module, in which a bar of the form $[a,b)$ in degree $d$ is counted with weight $(e^{-a}-e^{-b})$ and sign $(-1)^d$.  Persistent magnitude has good formal properties, such as additivity with respect to exact sequences and compatibility with tensor products, and has interpretations in terms of both the associated graded functor, and the Laplace transform.

Our definition is inspired by Otter's notion of blurred magnitude homology: we show that the magnitude of a finite metric space is precisely the persistent magnitude of its blurred magnitude homology.  Turning this result on its head, we obtain a strategy for turning existing persistent homology theories into new numerical invariants by applying the persistent magnitude.  We explore this strategy in detail in the case of persistent homology of Morse functions, and in the case of Rips homology.
\end{abstract}

\maketitle

\section{Introduction}

Magnitude is a numerical invariant of metric spaces arising from category
theory and with nontrivial geometric content.
In this paper we apply the theory of magnitude and its categorifications
to the study of persistence modules and persistent homology theories.

\subsection{Background}

Persistent homology, a fundamental notion in topological data analysis (or TDA),
is a tool for measuring the shape of data sets
and other objects.  The general idea is to take a data set and produce
an increasing sequence of topological spaces $X_s$, 
one for each value of a parameter $s$, where $X_s$ 
describes the shape of the data set `at scale $s$'.
Taking the homology of the $X_s$ produces the homology groups
$H_\ast(X_s)$ together with structure maps
$H_\ast(X_s)\to H_\ast(X_{s'})$ whenever $s\leqslant s'$.
This structure is called the \emph{persistent homology} of the
data set, and it is an example of an
algebraic structure called a graded \emph{persistence module}.
Any (sufficiently nice) persistence module has a \emph{barcode decomposition}
describing its isomorphism class in terms of a collection of intervals
called \emph{bars}.  Each bar is interpreted as a feature of the data
set: the start point of the interval is the scale at which
the feature first comes into being, and the end point is the scale at
which the feature evaporates. In one common interpretation, longer 
bars are interpreted as significant features, while shorter bars are
interpreted as noise.

\emph{Magnitude} is a numerical invariant of metric spaces introduced by
Leinster~\cite{LeinsterMetricSpace} (see also the survey \cite{survey}),
as an instance of a general category theoretical construction.
Despite its abstract origins, magnitude is a rich geometric invariant:
Meckes~\cite{MeckesMagnitudeDimensions} showed that magnitude can detect
the Minkowski dimension of compact subsets of Euclidean space,
Barcel\'o-Carbery~\cite{BarceloCarbery} showed that it can detect the
volume of compact subsets of Euclidean space,
and Gimperlein-Goffeng~\cite{GimperleinGoffeng} showed that it
can in addition detect surface area and the second intrinsic volume $V_2$
of appropriate subsets of odd-dimensional Euclidean space.

Magnitude of metric spaces has a categorification,
called \emph{magnitude homology}, which was introduced by 
Hepworth-Willerton~\cite{richard} and Leinster-Shulman~\cite{shulman}.
The magnitude homology of a metric space is a bigraded abelian group,
whose graded Euler characteristic recovers
the magnitude of the metric space, at least when the space is finite.  
Thus the relationship between
magnitude and magnitude homology is analogous to the relationship
between Euler characteristic and singular homology.
More recently, Otter~\cite{Otter} has introduced a \emph{blurred}
or persistent version of magnitude homology, which relates magnitude
homology to the Rips complex and, importantly, to ordinary homology.

\subsection{Results}

Blurred magnitude homology assigns to each metric space
$X$ a graded persistence module $\BMH_\ast(X)$.  
When $X$ is finite, 
we show that there is an attractive relationship between
the barcode decomposition of $\BMH_\ast(X)$ and the magnitude $|X|$ of $X$:
\[
	|X| 
	= 
	\sum_{k=0}^{\infty}
	\sum_{i=1}^{m_k}
	(-1)^k (e^{-a_{k,i}} - e^{-b_{k,i}})
\]
where $\BMH_\ast(X)$ has bars
$[a_{k,1},b_{k,1}),\ldots,[a_{k,m_k},b_{k,m_k})$
in degree $k\geq 0$.

Observe that the right hand side of the equation above
makes sense for any graded persistence module,
so long as it is subject to a finiteness condition such
as being finitely presented.
We turn this observation into a definition:
The \emph{persistent magnitude}
or simply \emph{magnitude} $|M_\ast|$ of a
finitely presented graded persistence module
$M_\ast$ is defined by
\[
	|M_\ast| 
	= 
	\sum_{k=0}^{\infty}
	\sum_{i=1}^{m_k}
	(-1)^k (e^{-a_{k,i}} - e^{-b_{k,i}})
\]
where $M_\ast$ has bars $[a_{k,1},b_{k,1}),\ldots,[a_{k,m_k},b_{k,m_k})$
in degree $k\geq 0$.
Note that a bar $[a,b)$ makes a contribution of
$\pm(e^{-a} - e^{-b})$ to the magnitude, so that longer bars
make a greater contribution, in line with one of the general
philosophies of persistent homology.

Persistent magnitude has good formal properties:
we show that 
it is additive with respect to exact sequences,
and that the magnitude of a tensor product of persistence
modules is the product of the magnitudes of the factors,
so long as the tensor product is understood in an appropriate
derived sense.

Now suppose that we have a persistent homology theory
defined for some class of mathematical objects,
for example the Rips homology of metric spaces.  By applying persistent
magnitude to the persistent homology, we obtain a new numerical
invariant of the mathematical objects in question.
Our first example of this process is the case of the sublevel set 
persistent homology of Morse functions, where the resulting
magnitude invariant is a (signed and weighted) count of the
critical points of the original function.  

Our most detailed example of persistent magnitude in action
is the \emph{Rips magnitude}.  This is the numerical invariant
of finite metric spaces obtained by taking the persistent
magnitude of the Rips homology, and is given 
by the weighted simplex-count
\[
	|X|_\mathrm{Rips} 
	= 
	\sum_{\emptyset\neq A\subseteq X}(-1)^{\#A-1}e^{-\diam(A)}.
\]
We compute the Rips magnitude of cycle graphs with their path,
Euclidean and geodesic metrics.  In each case they are determined by 
a number-theoretical
formula reminiscent of the sum of divisors function.

In the original setting of magnitude,
defining the magnitude of infinite
metric spaces is not straightforward: 
the simplest method is to
take the supremum of the magnitude of all finite subspaces of the
given infinite metric space, but there are alternatives, and currently
the theory only works well in the case of \emph{positive definite}
spaces.
We conclude the paper by investigating the question of whether
Rips magnitude
can be extended to infinite metric spaces.
In the case of closed intervals in $\R$ the approach
via a supremum works well and we find that $|[a,b]|_\mathrm{Rips}=1+(b-a)$.
In the case of the circle with its Euclidean and geodesic
metrics, which we study in detail, the results are attractive
but inconclusive.

\subsection{Organisation}
We begin the paper with a series of generous background sections:
persistence modules and persistent homology 
in section~\ref{section-background-persistence},
magnitude in section~\ref{section-background-magnitude},
and magnitude homology in section~\ref{section-background-mh}.
Section~\ref{section-persistent-magnitude} introduces the persistent
magnitude of persistence modules, and studies its basic properties.
Section~\ref{section-sublevel} applies persistent magnitude 
to the persistent homology of sublevel sets.
The final part of the paper studies Rips magnitude:
section~\ref{section-rips} introduces Rips magnitude and
discusses its properties and some basic examples,
section~\ref{section-cycles} computes it in the case
of cycle graphs (with various metrics),
and section~\ref{section-infinite} explores the possibility
of defining Rips magnitude for infinite metric spaces.

\subsection{Open Questions}

The results obtained in this work raise several natural questions, that we have not yet been able to answer conclusively:

\begin{itemize}
\item What is the most general notion of tameness sufficient to develop the theory of persistent magnitude? (Our characterisation using the Laplace transform suggests that one might want to consider a notion of persistence modules of ``exponential type'', meaning that the rank function is of exponential type.)
\item Is there a general definition of Rips magnitude for (a suitable class of) infinite metric spaces? Can we establish asymptotic results similar to the case of the circle for higher-dimensional spheres or other manifolds? (Note that not much seems to be known about Rips filtrations of manifolds beyond the circle \cite{adamaszek2017vietoris}.)
\item The formulas for the Rips magnitudes of cycle graphs and Euclidean cycles seem reminiscent of the sum of divisors functions from number theory. Are there interesting connections between Rips magnitude of cycles and analytic number theory?
\end{itemize}

\section{Background on persistence modules and persistent homology}
\label{section-background-persistence}

\subsection{Persistence modules}

Here we review some standard material on persistence modules, mostly following \cite{chazal2016structure,bubenik2014categorification}. For a survey explaining the basic ideas and historical origin of persistence, see \cite{edelsbrunner2008persistent}. A modern exposition of the main ideas including the structure and stability theorems for persistence modules can be found in \cite{chazal2016structure}. For further background on persistence modules from the category theoretical perspective, see \cite{bubenik2014categorification}. A slightly more algebraic perspective, with a view towards multi-parameter persistence, can be found in \cite{lesnick2015theory}. An account of some aspects of homological algebra for persistence modules can be found in \cite{BubenikMilicevic}.

Throughout the paper, we will work with vector spaces over a fixed field $\kk$.
The category of vector spaces over $\kk$ will be denoted by $\Vect$. In the most
general setting, persistence modules can be considered over an arbitrary small
category, see e.g.~\cite{botnan2018decomposition,bubenik2014categorification};
however, we will restrict attention to the case of $(\R,\leq)$-indexed persistence
modules, as this is entirely sufficient for our purposes. 
Here $(\R,\leq)$ denotes either the poset $\R$ equipped with the partial
order $\leq$, or the associated category with objects $\R$ and a unique
morphism $x\to y$ whenever $x\leq y$.

\begin{definition}
A {\em persistence module} is a functor $M\colon(\R,\leq)\to\Vect$. A {\em morphism} of persistence modules is a natural transformation of such functors.
\end{definition}

\begin{remark}
The category $\persmod=\Vect^{(\R,\leq)}$ of persistence modules has the structure of an abelian category. In particular, morphisms of persistence modules have well-defined kernels, and cokernels. These are again persistence modules and can be computed object-wise. The zero object of this abelian category is the persistence module $0\colon (\R,\leq)\to\Vect$ all of whose components are $0$.
\end{remark}

\begin{remark}
In some cases, we will also consider {\em graded persistence modules}, which are functors $M\colon (\R,\leq)\to\GrVect$, where $\GrVect$ is the category of $\N_0$-graded vector spaces over $\kk$. Most of the content of this section generalises to the graded case in a completely straightforward way, so to avoid too much duplication, we only state it for the ungraded case.
\end{remark}

To be able to extract any sort of useful information from persistence modules, we need to understand their structure. One way of doing this is by decomposing them into indecomposable summands. The indecomposables relevant in our case are known as interval modules.

\begin{definition}
A persistence module $M\colon(\R,\leq)\to\Vect$ is {\em indecomposable} if $M\cong M_1\oplus M_2$ implies that either $M_1\cong 0$ or $M_2\cong 0$.
\end{definition}

\begin{definition}
Let $J\subseteq\R$ be an interval. The {\em interval module} $\kk J\colon (\R,\leq)\to\Vect$ is defined as
\[
\kk J(x)=\begin{cases}
\kk;&\text{if $x\in J$,}\\
0;&\text{otherwise,}
\end{cases}
\]
and
\[
\kk J(x\leq y)=\begin{cases}
\id_{\kk};&\text{if $x,y\in J$,}\\
0;&\text{otherwise.}
\end{cases}
\]
\end{definition}

One of the main features of persistence modules over $(\R,\leq)$ that makes them useful in TDA
is that they can frequently be decomposed as direct sums of interval modules. When such a decomposition exists, it is unique \cite{azumaya1950corrections}. The following version of the decomposition theorem is originally due to Crawley-Boevey \cite{crawley2015decomposition}. In the case of persistence modules over $(\R,\leq)$, it can be stated as follows.

\begin{theorem}
Suppose $M\colon(\R,\leq)\to\Vect$ is a persistence module such that $M(x)$ is finite dimensional for every $x\in\R$. Then $M$ has a decomposition into interval modules.
\end{theorem}

Whenever a persistence module $M\colon(\R,\leq)\to\Vect$ decomposes as a sum of interval modules, we can represent it using a {\em persistence barcode}. This is defined as the multiset of all intervals that occur in the decomposition. Sometimes we represent these intervals as pairs $(a,b)$ where $a$ is the startpoint and $b$ is the endpoint of an interval in the decomposition. (These points are sometimes decorated to preserve information regarding which types of intervals the points correspond to, see \cite{chazal2016structure} for details.) The multiset of such pairs is called the {\em persistence diagram} corresponding to $M$. The notion of persistence diagram can be generalised to some cases where the interval decomposition does not exist \cite{chazal2016structure}.

We will often concentrate on the case of \emph{finitely presented}
persistence modules.  
Note that a persistence module is finitely presented if and only if it is isomorphic to a finite direct sum of half-open interval modules $\kk[a,b)$, where $-\infty<a<b\leq\infty$.

\begin{definition}
The \emph{tensor product} of two persistence modules $M$ and $N$ is given by
\[
(M\otimes N)(s) = \colim_{s_1+s_2\leq s}M(s_1)\otimes N(s_2).
\]
Thus $(M\otimes N)(s)$ is the quotient of
$\bigoplus_{s_1+s_2=s}M(s_1)\otimes M(s_2)$
obtained as follows.
Suppose given $u_1,u_2$ with $u_1+u_2\leq s$.
Then for any pair $v_1,v_2$ with $v_1+v_2=s$
and $u_1\leq v_1$, $u_2\leq v_2$,
we have a composite
\begin{equation}\label{equation-tensor}
	M(u_1)\otimes N(u_2)
	\to
	M(v_1)\otimes N(v_2)
	\hookrightarrow
	\bigoplus_{s_1+s_2=s}M(s_1)\otimes M(s_2).
\end{equation}
Then $(M\otimes N)(s)$ is the largest quotient
of $\bigoplus_{s_1+s_2=s}M(s_1)\otimes M(s_2)$
with the property that for all $u_1,u_2$
all such composites \eqref{equation-tensor} coincide,
regardless of the choice of $v_1,v_2$.
See Section~3.2 of~\cite{BubenikMilicevic}
or Section~2.2 of~\cite{PSS}.
\end{definition}

The operation of tensoring with a fixed persistence
module is right exact but not exact, and therefore
induces \emph{derived functors} denoted by
$M,N\mapsto\tor_i(M,N)$ for $i\geq 0$,
with $\tor_0(M,N)=M\otimes N$.

For finitely presented persistence modules
the tensor products and $\tor$-functors can be described explicitly.
In order to do this, it suffices to explain what happens for
interval modules.
Given interval modules $\kk[a,b)$ and $\kk[c,d)$, we have
\begin{align*}
	\kk[a,b)\otimes \kk[c,d) &= \kk[a+c,\min(a+d,b+c)),
	\\
	\tor_1(\kk[a,b), \kk[c,d)) &= \kk[\max(a+d,b+c),b+d),
	\\
	\tor_i(\kk[a,b), \kk[c,d)) &= 0
	\text{ for }i\geq 2.
\end{align*}
See Example~7.1 of~\cite{BubenikMilicevic}.
%Now we have the following.

\subsection{Persistent homology}
Persistence modules have an important role in TDA, where they are used
in order to study data sets in the form of finite metric spaces,
also known as \emph{point clouds}.  The idea is to take a finite metric space 
$X$ and convert it into a simplicial complex 
(or topological space or other topological object)
$Y$ equipped with an $(\R,\leq)$-filtration, 
i.e.~a system of subsets $Y_r\subseteq Y$ for $r\in\R$, 
such that $\bigcup_{r\in\R}Y_r=Y$ and $Y_r\subseteq Y_{r'}$ for $r<r'$.
There are many such constructions, and they are often based on the principle
that $Y_r$ should capture the behaviour of $X$ `at length scale $r$'.
Given such an $(\R,\leq)$-filtered complex $Y$,
the assignment $r\mapsto Y_r$ defines a functor
from $(\R,\leq)$ into simplicial complexes (or topological spaces,
or other appropriate codomain).  So taking the homology 
of the $Y_r$ then produces a graded persistence module 
\[
	r\longmapsto H_\ast(Y_r).
\]
These persistence modules are called the \emph{persistent homology} of 
the original object $X$.
Once the persistent homology of $X$ has been obtained,
the resulting barcode is then analysed.  The bars are regarded as features
of the metric space $X$. 
Longer or \emph{persistent} bars are often
regarded as genuine features, while shorter bars are often regarded as noise,
though there are other interpretations of the barcode.

Here we will describe some important examples of this general construction,
starting with the Vietoris-Rips filtration and the \v{C}ech filtration. 

\begin{definition}
Suppose $(X,d)$ is a finite metric space. We define the {\em Vietoris-Rips complex $\calR(X)$} of $(X,d)$ to be the $(\R,\leq)$-filtered simplicial complex with vertex set $X$, in which the simplices of the $r$-th filtration step $\calR_r(X)$ are defined by the rule
\[
\sigma\in\calR_r(X)\Leftrightarrow\diam\sigma\leq r.
\]
\end{definition}

In some cases, we consider $X$ as a subspace of some larger metric space $Y$, e.g.~$Y=\R^n$. In this case we can define the corresponding \v{C}ech complex as follows:

\begin{definition}
Suppose $(Y,d)$ is a metric space and $X\subseteq Y$ is a finite subset. We define the {\em \v{C}ech complex $\calCh(X)$} associated to $X$ to be the $(\R,\leq)$-filtered simplicial complex with vertex set $X$, in which the simplices of the $r$-th filtration step $\calCh_r(X)$ are defined by the rule
\[
\sigma\in\calCh_r(X)\iff\bigcap_{x\in\sigma}B(x,r)\neq\emptyset,
\]
where $B(x,r)$ denotes the open ball in $Y$ with centre $x$ and radius $r$.
\end{definition}

Note that both the Vietoris-Rips and the \v{C}ech complex are filtrations of the simplex spanned by the vertices of $X$.

A related source of persistence modules are sublevel set filtrations. These are associated to a function $f\colon X\to\R$, where $X$ is a topological space. They are motivated by ideas of Morse theory, where $X=M$ is assumed to be a smooth manifold and $f$ is a Morse function (smooth function whose critical points are nondegenerate).

\begin{definition}
Let $f\colon X\to\R$ be a (continuous) function on a topological space $X$. The {\em sublevel set filtration} associated to $(X,f)$ is the family $(X^a)_{a\in\R}$ where $X^a=f^{-1}(-\infty,a]$, which can also be viewed as a functor $S\colon (\R,\leq)\to\Top$. Composing this functor with $k$-th singular homology yields a persistence module $H_k\circ S$ which is called the {\em $k$-th sublevel set persistent homology of $(X,f)$}.
\end{definition}

Other examples of persistence modules that have been used are lower
star filtrations of simplicial complexes, alpha (or Delaunay) complexes, wrap complexes,
witness complexes, and many more besides \cite{edelsbrunner2010computational,de2004topological,bauer2017morse}.

In order to ensure stability of persistence modules arising in applications
despite the noise arising from imprecise measurements, it is important to be able to use
approximation techniques. This is done using the notion of $\epsilon$-interleavings.
These provide a way to formalise the intuitive notion of approximate isomorphism
of persistence modules and can be used to define a notion of distance on the category
of persistence modules. For details, see \cite{chazal2016structure,bubenik2014categorification,lesnick2015theory}.

\section{Background on magnitude of metric spaces}
\label{section-background-magnitude}

In this section we will introduce the magnitude of metric spaces.
This is a numerical invariant of metric spaces developed by Tom Leinster
in~\cite{LeinsterMetricSpace}, building on earlier work defining numerical
invariants of categories~\cite{LeinsterEulerCharCategory}.  Despite
these abstract origins, magnitude turns out to be an interesting invariant containing meaningful geometric information.
Here we will introduce the basics and attempt to give readers
an impression of magnitude's interest and reach.
Readers who wish to know more are strongly recommended to take a look 
at Leinster's original paper~\cite{LeinsterMetricSpace}
and Leinster and Meckes's survey~\cite{survey}.
We note here that magnitude of metric spaces is just one instance
of a more general invariant of enriched categories.
The latter is developed in section~1 of~\cite{LeinsterMetricSpace},
and we will not say anything about it here.

Here, and in the rest of the paper, we will use the symbol
$|X|$ to denote the magnitude of an object $X$.
To avoid notational clashes, we will use the symbol
$\#X$ to denote the cardinality of a finite set $X$.

\subsection{Magnitude of finite metric spaces}

We begin with the magnitude of finite metric spaces.
This is based almost entirely on section~2 of~\cite{LeinsterMetricSpace}.

\begin{definition}[Magnitude via weightings]
	Let $(X,d)$ be finite metric space. 
	A \emph{weighting} on $X$ is a function $w\colon X\to\R$ 
	such that the equality
	\[
		\sum_{y\in X}e^{-d(x,y)}w(y)=1
	\]
	is satisfied for every $x\in X$. 
	If $X$ admits a weighting, then we define the \emph{magnitude} 
	of $X$ to be 
	\[
		|X|=\sum_{x\in X}w(x).
	\]
	This is independent of the choice of weighting.
	If no weighting exists, then the magnitude of $X$ is not defined.
\end{definition}

\begin{remark}[Magnitude via matrices]
	Suppose that $X$ is a finite metric space with elements
	$x_1,\ldots,x_n$, and let
	$Z_X$ denote the $n\times n$ matrix with
	$(Z_X)_{ij}=e^{-d(x_i,x_j)}$.  
	If it happens that $Z_X$ is invertible,
	then the magnitude of $X$ is defined and is given by the formula
	\begin{equation}\label{equation-inverse}
		|X| = \sum_{i,j=1}^n (Z_X^{-1})_{ij}.
	\end{equation}
	It can happen that $|X|$ is defined (using weightings)
	in cases where $Z_X$ is not invertible.
	(See Lemma~1.1.4 of~\cite{LeinsterMetricSpace}.)
\end{remark}

For $t>0$, we let $tX$ be the metric space $X$ rescaled by $t$, 
so that $d_{tX}(x,y)=td_X(x,y)$.  
There is no simple relationship between $|tX|$ and $|X|$,
and as a consequence we gain information
by considering all rescalings at once, as in the following definition.

\begin{definition}[The magnitude function]
\label{definition-magnitude-function}
	Let $X$ be a finite metric space.
	Its \emph{magnitude function} is the (partially defined)
	function from $(0,\infty)$ to $\R$ given by
	\[
		t\mapsto |tX|.
	\]
\end{definition}

\begin{example}[Magnitude of the one-point space]
	Let $X$ denote the space consisting of a single point $x$.
	Then $Z_X$ is the $1\times 1$ matrix $(1)$,
	so that $Z_X^{-1}=(1)$ and formula~\eqref{equation-inverse}
	gives us $|X| = 1$.
\end{example}

\begin{example}[Magnitude of two-point spaces]	
	Let $X=\{x_1,x_2\}$ be the two-point space in which
	$d_X(x_1,x_2)=d$ for some $d>0$.
	Then
	\[
		Z_X = \begin{pmatrix} 1 & e^{-d} \\ e^{-d} & 1\end{pmatrix}
	\]
	so that 
	\[
		Z_X^{-1} = \frac{1}{1-e^{-2d}}
		\begin{pmatrix} 1 & -e^{-d} \\ -e^{-d} & 1  \end{pmatrix}
	\]
	and consequently
	\[
		|X|=\frac{2-2e^{-d}}{1-e^{-2d}}=\frac{2}{1+e^{-d}}.
	\]
	The same computation shows that the magnitude function of $X$
	is given by
	\[
		|tX|=\frac{2}{1+e^{-dt}}
	\]
	with graph:
	\begin{center}
	\includegraphics[width=200pt]{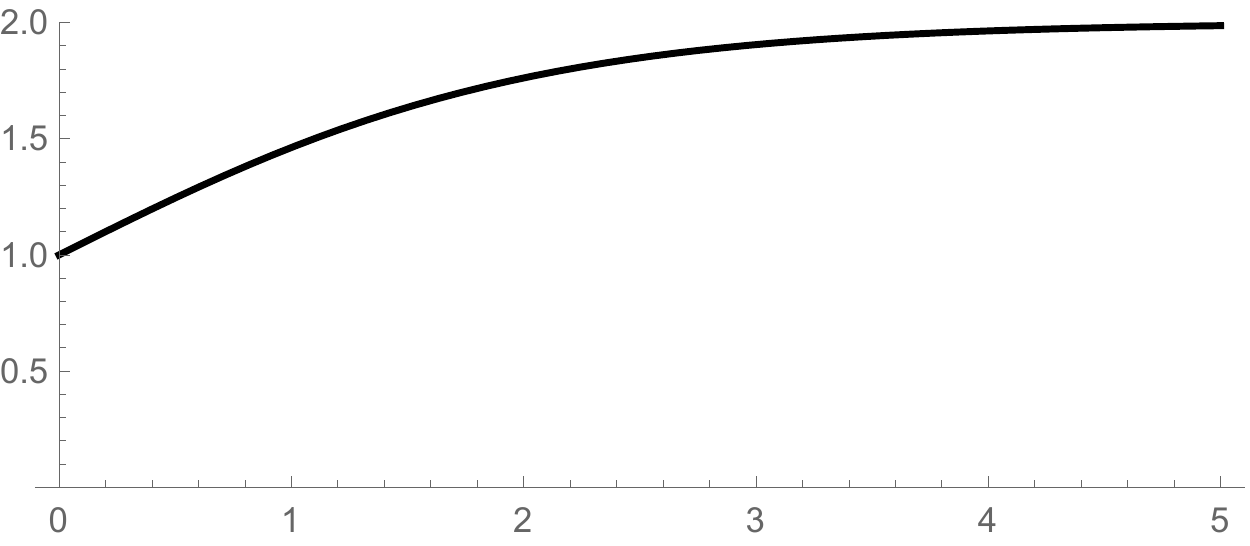}
	\end{center}
	We see in this case that $|tX|$
	varies between $1$ and $2$, tending to $1$ as $t\to 0$
	and to $2$ as $t\to \infty$.
	This suggests that magnitude is an `effective number of points',
	regarding two points as essentially the same if they are very close,
	and essentially different if they are very far apart.
	The latter property generalises.
\end{example}

\begin{proposition}[{Leinster~\cite[Proposition~2.2.6]{LeinsterMetricSpace}}]
	Let $X$ be a finite metric space.
	Then $|tX|\to \#X$ as $t\to\infty$,
	where $\#X$ denotes the cardinality of $X$.
\end{proposition}

\begin{example}
\begin{itemize}
	\item
	It is known that all metric spaces with four points
	or less have magnitude, but there exist spaces with five or more
	points that do not have magnitude
	(See pages 870-871 of~\cite{LeinsterMetricSpace}).

	\item
	There is a simple formula due to Speyer 
	for the magnitude of \emph{homogeneous}
	metric spaces, i.e.~those that admit a transitive group action
	(see Proposition~2.1.5 of~\cite{LeinsterMetricSpace}).
	This allows one to compute magnitude of many
	simple spaces, for example complete graphs and cyclic graphs.
	(Graphs are always regarded as metric spaces
	by equipping them with the shortest path metric.)

	\item
	The magnitude function of a finite metric space $X$
	can take negative values, it can take values greater than $\#X$,
	and it can have intervals on which it is increasing
	or decreasing.  Example~2.2.7 of~\cite{LeinsterMetricSpace}
	gives a demonstration of this on a space $X$ with 5 points.
	
	\item
	It is not always true that $|tX|\to 1$ as $t\to 0$.
	An example due to Willerton describes a metric space
	with $6$ points for which $|tX|\to 6/5$ as $t\to 0$.
	(See Example~2.2.8 of~\cite{LeinsterMetricSpace}.)
\end{itemize}
\end{example}

\begin{remark}[Magnitude and data]
	One may take a data set in the form of 
	a finite subspace of Euclidean space, 
	and take its magnitude or magnitude function, which in this case
	is always defined.
	The result is a potentially interesting invariant of such data sets.
	But for this to be useful, one would like to know that
	the invariant is stable under perturbations of the data set.
	In mathematical terms, one would like to know that magnitude
	is continuous with respect to the Hausdorff metric on subsets
	of Euclidean space.
	This is currently unknown, although Meckes has shown
	that in this situation the function $X\mapsto |X|$ is 
	\emph{lower semicontinuous},
	meaning roughly that magnitude may jump upwards but not downwards.
	(See Theorem~2.6 of~\cite{MeckesPositiveDefinite}
	and the paragraph that follows it.)
\end{remark}

\subsection{Magnitude of compact metric spaces}

Magnitude also makes sense for certain classes of compact,
infinite metric spaces.  Here we will recall the relevant definition
and some of the main results.  Good references for this section
are section~3 of~\cite{LeinsterMetricSpace} and the survey~\cite{survey}.

In the following we will consider \emph{positive definite}
metric spaces, which are metric spaces $X$ with the property that 
for every finite subspace $F$ 
the matrix $Z_F$ is positive definite.
Any subset of Euclidean space, with its induced metric, is positive
definite.

\begin{definition}
Let $(X,d)$ be a compact positive definite metric space. \emph{The magnitude} of $X$ is defined by the formula:
\[
|X|=\sup\{|W| : W\subseteq X,\ W\text{ finite.}\}
\]
The \emph{magnitude function} of $X$ is defined by $t\mapsto |tX|$
for $t\in [0,\infty)$.
\end{definition}

\begin{example}
We let $S^1_\mathrm{eucl}$ denote the Euclidean circle, i.e.~the unit circle
in the plane with its induced metric.  And we let $S^1_\mathrm{geo}$ denote
the same circle with its geodesic metric of total arclength $2\pi$.
Both are positive definite.
Then the magnitude function of the Euclidean circle is given by
\[
|t\cdot S^1_\mathrm{eucl}|=\pi t+O(t^{-1})\qquad\text{as}\qquad t\to\infty
\]
and the magnitude function of the geodesic circle is given by
\[
|t\cdot S^1_{\mathrm{geo}}|=\frac{\pi t}{1-e^{-\pi t}}.
\]
See Theorems~13 and~14 of~\cite{asymptotic}. It has been argued \cite{asymptotic}
that the linear term $\pi t$ in these expressions corresponds to half the length
of the circle, whereas the absence of the constant term corresponds to the fact
that the Euler characteristic of the circle is zero.
\end{example}

\begin{remark}
	In general, computing the magnitude of infinite spaces is difficult,
	and existing computations tend to require a significant 
	amount of analysis.
	A useful survey on this subject is given in~\cite{survey}.
	Important recent progress by Gimperlein and 
	Goffeng~\cite{GimperleinGoffeng}
	shows that for appropriate $X\subseteq\R^{2n+1}$,
	the asymptotics of the magnitude function 
	as $t\to\infty$ encode geometric properties including volume, 
	surface area and mean curvature.
\end{remark}

\section{Background on magnitude homology}
\label{section-background-mh}

Singular homology can be regarded as a \emph{categorification}
of the Euler characteristic: The Euler characteristic is a function
taking values in the set of integers, 
whereas homology is a functor taking values in the category
of graded abelian groups, and the 
function can be obtained from the functor by taking the alternating
sum of the ranks:
\[
	\chi(X) = \sum_{i=0}^\infty (-1)^i \rank H_i(X)
\]
This is a classical story, but 
there are more recent examples of such categorifications,
notably Khovanov homology, which categorifies the Jones polynomial,
and Knot Floer homology, which categorifies the Alexander polynomial.

Hepworth and Willerton~\cite{richard}
together with Leinster and Shulman~\cite{shulman}
introduced \emph{magnitude homology}, a categorification of
magnitude.  (Precisely, Hepworth and Willerton first introduced
magnitude homology in the case of graphs, 
and Leinster and Shulman later extended
this to arbitrary metric spaces and very general enriched categories.)
More recently, Nina Otter~\cite{Otter} introduced a persistent
version of magnitude homology called \emph{blurred magnitude homology}.

In this section we will introduce magnitude homology
and its blurred variant, and we will conclude by giving
an explicit formula to extract the magnitude of a space
from the barcode of its blurred magnitude homology.
It is this story that we will \emph{reverse} in the
rest of the paper, using its conclusion as the \emph{definition}
of the magnitude of persistence modules.
Applying this to persistent homology theories other than
blurred magnitude homology, we will then 
obtain new notions of magnitude of metric spaces.

\subsection{Magnitude homology}
Given a metric space $X$ and elements $x_0,\ldots,x_k\in X$,
we define
\[
	\ell(x_0,\ldots,x_k)
	=
	d(x_0,x_1)+d(x_1,x_2)+\cdots+d(x_{k-1},x_k).
\]
We think of this as the \emph{length} of the tuple $(x_0,\ldots,x_k)$.

\begin{definition}[Hepworth-Willerton~\cite{richard},
Leinster-Shulman~\cite{shulman}]\label{definition-mh}
	The \emph{magnitude chain complex} of a 
	metric space $X$ consists of the abelian groups 
	\[
		\MC_{k,l}(X)
		=
		\left\langle
			(x_0,\ldots,x_k)\in X^{k+1}
			\ \middle|\ 
			\begin{array}{l}
			x_{0}\neq x_1\neq\cdots\neq x_k,\\
			\textstyle \ell(x_0,\ldots,x_k)=l
			\end{array} 
		\right\rangle
	\]
	with $l\in[0,\infty)$ and $k$ a non-negative integer.
	Here, and in what follows, angled brackets $\langle\quad\rangle$
	denote free $\Z$-modules.
	The boundary operators 
	\[
		\partial_{k,l}\colon \MC_{k,l}(X)\to \MC_{k-1,l}(X)
	\]
	are defined by the rule
	\[
		\partial_{k,l}(x_0,\ldots,x_k)
		=
		\sum_{i=0}^k (-1)^i(x_0,\ldots,\widehat{x_i},\ldots,x_k),
	\]
	where the term $(x_0,\ldots,\widehat{x_i},\ldots,x_k)$
	is omitted if 
	$\ell(x_0,\ldots,\widehat{x_i},\ldots,x_k)<l$.
	The \emph{magnitude homology} $\MH_{k,l}(X)$ of $X$ is defined to
	be the homology of the magnitude chains
	\[
		\MH_{k,l}(X)=H_k(\MC_{\ast,l}(X))
	\]
	where again $k$ is a non-negative integer and $l\in[0,\infty)$.
\end{definition}

Magnitude homology is a categorification of the magnitude,
in the sense that the graded Euler characteristic
of magnitude homology coincides with the magnitude itself,
as shown in the next proposition.
This is categorification in the same sense that Khovanov homology
categorifies the Jones polynomial, and that knot Floer homology
categorifies the Alexander polynomial.

\begin{proposition}[Hepworth-Willerton~\cite{richard},
Leinster-Shulman~\cite{shulman}]
\label{proposition-alternating}
	Let $X$ be a finite metric space. Then
	\[
		|tX|=\sum_{l\in[0,\infty)}
		\sum_{k=0}^\infty
		(-1)^k\rank(\MH_{k,l}(X))e^{-lt}
	\]
	for $t$ sufficiently large.
\end{proposition}

\begin{remark}\label{remark-l-values}
	The formula above requires some elaboration.
	Consider the set of real numbers occuring as
	$\ell(x_0,\ldots,x_k)$ for $x_0,\ldots,x_k\in X$,
	$k\geq 0$, with consecutive $x_i$'s distinct.
	Let us call these \emph{$\ell$-values}.
	Since $X$ is finite, 
	there is a positive minimum nonzero distance between elements of $X$,
	call it $\delta>0$, and then 
	all $\ell$-values satisfy the following inequality:
	\begin{equation}\label{equation-ell-bound}
		\ell(x_0,\ldots,x_k)\geq \delta k
	\end{equation}

	A first consequence of equation~\eqref{equation-ell-bound}
	is that, for a fixed choice of $l\in[0,\infty)$,
	the set of $k$ for which $\MH_{k,l}(X)\neq 0$ is bounded
	above by $l/\delta$.  
	That is because if $\MH_{k,l}(X)\neq 0$ then $l$ must be an
	$\ell$-value $\ell(x_0,\ldots,x_k)$.
	It follows that in Proposition~\ref{proposition-alternating}
	the inner sum is finite for each $l$.

	The second consequence of equation~\eqref{equation-ell-bound}
	is that for any positive real $N$, the collection of
	$\ell$-values satisfying $\ell(x_0,\ldots,x_k)\leq N$
	is finite (because then $k\leq N/\delta$, and $X$ is finite).
	It follows that the set of all $\ell$-values
	can be totally ordered $0=l_0<l_1<l_2<\cdots$.
	Thus the outer series 
	in Proposition~\ref{proposition-alternating} 
	can be rewritten as the (infinite) sum over the $l_i$.
\end{remark}	

\begin{example}[Magnitude homology of graphs]
	A graph can be regarded as a metric space 
	by taking the set of vertices and equipping them with
	the shortest path metric.  This is the original setting
	of magnitude homology in~\cite{richard},
	where a number of explicit examples 
	(done using computer algebra) are
	described.  
	We include two of these here 
	as an illustration.
	Figure~\ref{TableFiveCycle} shows the ranks of the
	magnitude homology $\MH_{k,l}(C_5)$ of the cyclic graph
	with $5$ vertices,
	and Figure~\ref{TablePetersen} shows the ranks of 
	the magnitude homology $\MH_{k,l}(\mathit{P\!etersen})$ of 
	the Petersen graph.
	(Note that the images and tables in Figures~\ref{TableFiveCycle}
	and~\ref{TablePetersen} are taken directly from~\cite{richard}.)
	Observe that in each case, the rank of $\MH_{0,0}(G)$
	is the number of vertices, and the rank of $\MH_{1,1}(G)$
	is the number of oriented edges.  These are general features,
	but the question of what data is encoded in $\MH_{k,l}(G)$
	for other choices of $k$ and $l$ remains mysterious.
	Another general feature visible here is that 
	the nonzero magnitude homology
	groups lie in a range of pairs $(k,j)$ bounded by two
	diagonals, one of them the diagonal $k=j$, and the
	other determined by the diameter of the graph.
	
	\begin{figure}
	\begin{center}
	\begin{tikzpicture}[scale=0.5, baseline=0]
	\foreach \x in {0,72,...,288}
	    \draw (\x+90:2cm) -- (\x+72+90:2cm);
	%\draw (0:1cm) -- (180:1cm);
	\foreach \x in {0,72,...,288}
	    \draw [fill=red](\x+90:2cm) circle (0.1cm);
	\end{tikzpicture}
	%
	%\quad
	%
	\footnotesize
	\begin{tabular}{rrrrrrrrrrrrrr}
	&&&&&&&$k$\\
	&&0&1&2&3&4&5&6&7&8&9&10&11\\
	\otoprule                                                                 
	&0 & 5\\
	& 1   & &     10                   \\                                          
	 &2     & &&         10             \\                                          
	& 3    &&&           10  &  10       \\                                          
	& 4         &&&&            30  &  10  \\                                         
	& 5          &&&&&                 50   & 10                          \\           
	$l$& 6          &&&&&                 20  &  70   & 10                    \\           
	& 7           &&&&&&                      80  &  90  &  10              \\           
	& 8           &&&&&&&                           180  & 110 &   10      \\             
	& 9            &&&&&&&                          40 &  320  & 130  &  10 \\            
	&10           &&&&&&&&                                 200 &  500  & 150 &   10 \\      
	&11           &&&&&&&&&                                       560  & 720 &  170  &  10\\ 
	\bottomrule
	\smallskip
	\end{tabular}
	\end{center}
	\caption{The ranks of $\MH_{k,l}(C_5)$.
	(Taken from~\cite[Table 1]{richard})}
	\label{TableFiveCycle}
	\end{figure}
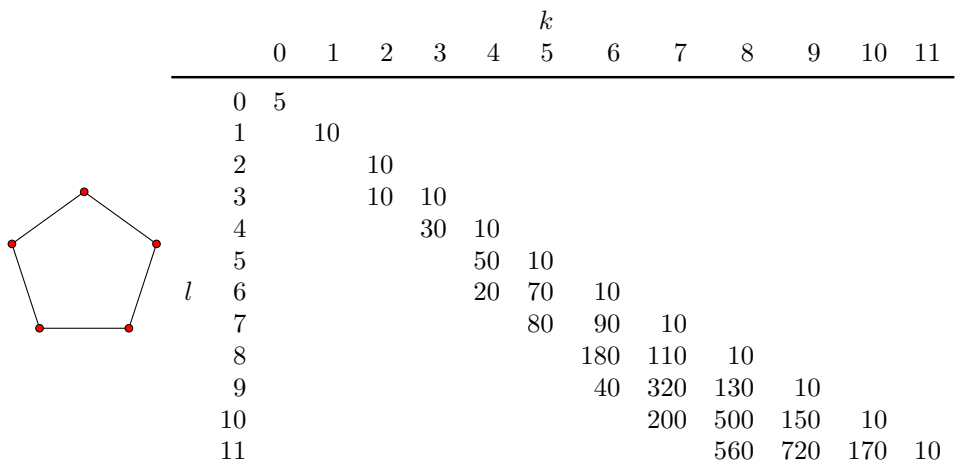

	\begin{figure}
	\begin{center}
	\begin{tikzpicture}[baseline=0cm]
	  \tikzset{VertexStyle/.style= {shape=circle,  color=black, fill=red, inner sep=0pt, 
		       minimum size = 0.1cm, draw}}
	   \SetVertexNoLabel
	   \tikzset{EdgeStyle/.style= {thick}} 
	   \grPetersen[form=1,RA=1,RB=0.6]%
	\end{tikzpicture}\quad
	\footnotesize
	\begin{tabular}{rrrrrrrrrrr}
	&&&&&$k$\\
	&0&1&2&3&4&5&6&7&8\\
	\otoprule                                                                 
	& 0 & 10\\
	&  1   & &     30                   \\                                          
	&  2     & &&         30             \\                                          
	&  3    &&&           120  &  30       \\                                          
	$l$ &  4         &&&&            480  &  30  \\                                         
	&  5          &&&&&                 840   & 30                          \\           
	&  6          &&&&&                 1440  &  1200   & 30                    \\           
	&  7           &&&&&&                      7200  &  1560  &  30              \\           
	&  8           &&&&&&&                           17280  & 1920 &   30      \\              
	\bottomrule
	\smallskip
	\end{tabular}
	\end{center}
	\caption{The ranks of $\MH_{k,l}(\mathit{P\!etersen})$.
	(Taken from~\cite[A.3.1]{richard}.) }
	\label{TablePetersen}
	\end{figure}
\end{example}

Magnitude homology has many good characteristics of homology theories
and categorification:
\begin{itemize}
	\item
	Magnitude homology refines magnitude:
	there are finite metric spaces with the same
	magnitude but non-isomorphic magnitude homologies~\cite{Gu}.

	\item
	Magnitude homology can contain 
	torsion~\cite{KanetaYoshinaga}.
	Thus the magnitude homology contains more data
	than just the ranks $\rank(\MH_{k,l}(X))$.

	\item
	Magnitude homology has properties that categorify
	known properties of the magnitude.
	In the setting of graphs,
	a K\"unneth theorem
	categorifies the known product formula for magnitude,
	and a Mayer-Vietoris sequence categorifies the known
	inclusion-exclusion formula.
	See~\cite{richard}.
	
	\item
	Magnitude homology contains information about geometric
	features of a metric space, for example it can precisely detect 
	the property of being Menger convex, and it contains obstructions
	to the existence of upper bounds on curvature, and to the existence
	of closed geodesics. See~\cite{shulman}, \cite{Gomi}, \cite{Asao}.

	\item
	Magnitude homology has been computed fully in several
	interesting examples, including trees, complete graphs,
	cycle graphs, and the icosahedral graph.
	See~\cite{richard} and~\cite{Gu}.
\end{itemize}

\subsection{Blurred magnitude homology}

We now describe some recent work of Nina Otter~\cite{Otter}
that connects magnitude homology with persistent homology,
specifically the Vietoris-Rips complex. We also give a new
result that relates magnitude with barcodes for the first time.

\begin{definition}[Otter~\cite{Otter}]\label{definition-bmh}
	The \emph{blurred magnitude chain complex} of a 
	metric space $X$ is the chain complex of persistence
	modules $\BMC_\ast(X)$ defined by the rule
	\[
		\BMC_k(X)(l)
		=
		\left\langle
			(x_0,\ldots,x_k)\in V^{k+1}
			\ \middle|\ 
			\begin{array}{l}
			x_{0}\neq x_1\neq\cdots\neq x_k,\\
			\textstyle \ell(x_0,\ldots,x_k)\leq l
			\end{array} 
		\right\rangle
	\]
	where $l$ is the persistence parameter
	and $k$ is a non-negative integer.
	The boundary operators 
	\[
		\partial_{k}\colon \BMC_k(X)\to \BMC_{k-1}(X)
	\]
	are defined by the rule
	\[
		\partial_{k,l}(x_0,\ldots,x_k)
		=
		\sum_{i=0}^k (-1)^i(x_0,\ldots,\widehat{x_i},\ldots,x_k).
	\]
	The \emph{blurred magnitude homology} $\BMH_\ast(X)$ 
	of $X$ is defined to be the homology of the blurred magnitude chains:
	\[
		\BMH_k(X)=H_k(\BMC_\ast(X))
	\]
	for $k$ a non-negative integer.
\end{definition}

\begin{remark}[Blurred magnitude homology and the Rips complex]
	One of the main results of Otter's paper~\cite{Otter}
	is that it compares
	the blurred magnitude homology of a metric space $X$ with 
	the homology of its Rips complex.
	The main idea of this comparison is that there are maps
	\[
		\BMC_k(X)(s)\to C_k(\calR^\mathrm{sim}(X)(s))
		\text{\ \ and\ \ }
		C_k(\calR^\mathrm{sim}(X)(s))\to \BMC_k(X)(ks),
	\]
	where $\calR^\mathrm{sim}(X)$ denotes a 
	variant of the Rips chain complex, 
	having the same persistent homology.
	These comparison maps are a multiplicative version of an interleaving,
	and although the constant appearing here is the degree $k$ in the chain
	complex, and in particular is not constant,
	it is nevertheless sufficient for Otter to prove the 
	following theorem.
\end{remark}

\begin{theorem}[{Otter~\cite[Theorem~32]{Otter}}]
\label{theorem-nina}
	\[
		\lim_{0\leftarrow\epsilon}\BMH_\ast(X)(\epsilon)
		\cong
		\lim_{0\leftarrow\epsilon}H_\ast(\calR(X)(\epsilon))
	\]
\end{theorem}

The quantity $\lim_{0\leftarrow\epsilon}H_\ast(\calR_\ast(X)(\epsilon))$
is the \emph{Vietoris homology} of $X$, a version of homology
developed for metric spaces. In good cases, e.g.~when $X$ is a
compact Riemannian manifold, it coincides with the singular homology
of $X$.  This theorem therefore demonstrates for the first time
a concrete connection between magnitude homology and ordinary
homology of spaces.

We will now state a new result that gives the relation between magnitude
and the barcode decomposition of the blurred magnitude homology.
First note that by standard homological algebra, the definitions stated above
have the following immediate consequence, which relates ordinary and blurred
magnitude homology. 

\begin{proposition} \label{proposition-absolute-blurred}
	Let $X$ be a finite metric space and let $0=l_0<l_1<l_2<\cdots$
	be the distinct real numbers occuring as 
	$\ell(x_0,x_1,\ldots,x_k)$ for $x_0,\ldots,x_k\in X$,
	$k\geq 0$.
	Then for each $k\geq 0$ and $j>0$ we have a short exact sequence:
	\[
	0\to\BMC_k(X)(l_{j-1})\to\BMC_{k}(X)(l_j)\to\MC_{k,l_j}(X)\to0
	\]
	Consequently in homology there is a long exact sequence:
	\[
	\cdots\to\BMH_k(X)(l_{j-1})\to\BMH_k(X)(l_j)\to\MH_{k,l_j}(X)\to\cdots
	\]
\end{proposition}

Our result can then be stated as follows. Its proof is rather long
and technical, thanks to convergence issues.

\begin{theorem}\label{theorem-magBMH}
	Let $X$ be a finite metric space and let $\BMH_\ast(X)$ denote
	its blurred magnitude homology.  Suppose that $\BMH_\ast(X)$
	has barcode whose bars in degree $k\geq 0$
	are $[a_{k,0},b_{k,0}),[a_{k,1},b_{k,1}),\ldots$.
	Then the magnitude of $X$ is given by the formula
	\[
		|tX| 
		= 
		\sum_{k=0}^\infty
		\sum_{i=1}^{m_k}
			(-1)^k (e^{-a_{k,i}t} - e^{-b_{k,i}t})
	\]
	for $t$ sufficiently large.
\end{theorem}

\begin{proof}
Throughout this proof we let 
$\delta$ denote the minimum nonzero distance
between elements of $X$, and we let $n$ denote the
cardinality of $X$.
We let $l_0<l_1<l_2<\cdots$ be the distinct values of $l$ for which
the inner sum appearing in Proposition~\ref{proposition-alternating}
is nonzero, as in Remark~\ref{remark-l-values}.  
And we define $D(i,j,k)$ to be $1$ if $l_j$ is in $[a_{k,i},b_{k,i})$,
and to be $0$ otherwise. 
We make a standing assumption that $t$ is large enough that
$ne^{-\delta t}<1$; this is the assumption under which
$t$ is large enough that the conclusions of 
Proposition~\ref{proposition-alternating} hold.

We will use the following fact several times.
Let $(x_0,\ldots,x_k)$ be a tuple of elements of $X$
in which consecutive elements are distinct,
and suppose that this tuple is a generator of 
$\MC_{k,l}(X)$ or $\BMC_k(X)(l)$.
Then $\ell(x_0,\ldots,x_k)\leq l$,
while $\ell(x_0,\ldots,x_k)\geq k\delta$,
so that $k\delta \leq l$.  
It follows that, if $k$ and $l$ do not satisfy this relation,
then the homology groups $\MH_{k,l}(X)$ and $\BMH_k(X)(l)$ vanish.

We now have the following computation, whose steps
will be justified below.
\begin{align*}
	|tX|
	&\stackrel{1}{=}
	\sum_{j=0}^\infty
	\sum_{k=0}^\infty
	(-1)^k\rank(\MH_{k,l_j}(X))e^{-l_jt}
	\\
	&\stackrel{2}{=}
	\sum_{j=0}^\infty
	\sum_{k=0}^\infty
	(-1)^k
	\left[
		\rank(\BMH_{k}(X)(l_j))
		-
		\rank(\BMH_{k}(X)(l_{j-1}))
	\right]e^{-l_jt}
	\\
	&\stackrel{3}{=}
	\sum_{k=0}^\infty
	(-1)^k
	\sum_{j=0}^\infty
	\left[
		\rank(\BMH_{k}(X)(l_j))
		-
		\rank(\BMH_{k}(X)(l_{j-1}))
	\right]e^{-l_jt}
	\\
	&\stackrel{4}{=}
	\sum_{k=0}^\infty
	(-1)^k
	\sum_{j=0}^\infty
	\rank(\BMH_{k}(X)(l_j))(e^{-l_jt}-e^{-l_{j+1}t})
	\\
	&\stackrel{5}{=}
	\sum_{k=0}^\infty
	(-1)^k
	\sum_{j=0}^\infty
	\sum_{i=0}^\infty
	D(i,j,k)(e^{-l_jt}-e^{-l_{j+1}t})
	\\
	&\stackrel{6}{=}
	\sum_{k=0}^\infty
	(-1)^k
	\sum_{i=0}^\infty
	\sum_{j=0}^\infty
	D(i,j,k)(e^{-l_jt}-e^{-l_{j+1}t})
	\\
	&\stackrel{7}{=}
	\sum_{k=0}^\infty
	(-1)^k
	\sum_{i=0}^\infty
	(e^{-a_{k,i}t}-e^{-b_{k,i}t})
\end{align*}

Step 1 is precisely the formula of 
Proposition~\ref{proposition-alternating}.
The series here is absolutely convergent.
That is because 
\begin{align*}
	\sum_{j=0}^J
	\sum_{k=0}^\infty 
	\rank(\MH_{k,l_j}(X))
	e^{-l_jt}
	&=
	\sum_{k=0}^\infty 
	\sum_{j=0}^J
	\rank(\MH_{k,l_j}(X))
	e^{-l_jt}
	\\
	&\leq
	\sum_{k=0}^\infty 
	\sum_{x_0,\ldots,x_k}
	e^{-\ell(x_0,\ldots,x_k)t}
	\\
	&\leq
	\sum_{k=0}^\infty
	n^{k+1}e^{-k\delta t}
	\\
	&=
	n\cdot\sum_{k=1}^\infty
	(ne^{-\delta t})^k.
\end{align*}
Here, in the second line the inner sum is over all tuples
$(x_0,\ldots,x_k)$ with consecutive elements distinct,
and there are at most $n^{k+1}$ of these, where $n$
denotes the cardinality of $X$.
Now we have
$ne^{-\delta t}<1$ by our standing assumption, 
so that the latter sum converges and is bounded above
independent of $J$.  This shows absolute convergence.

To explain step 2,
%note from the definitions
%that there is a short exact sequence
%\[
	%0
	%\to
	%\BMC_k(X)(l_{j-1})
	%\to 
	%\BMC_{k}(X)(l_j)
	%\to
	%\MC_{k,l_j}(X) 
	%\to
	%0
%\]
%and consequently
recall that by Proposition \ref{proposition-absolute-blurred} there is a long exact sequence
\[
	\cdots
	\to
	\BMH_k(X)(l_{j-1})
	\to
	\BMH_k(X)(l_j)
	\to
	\MH_{k,l_j}(X) 
	\to
	\cdots
\]
This sequence terminates in both directions,
because the relation described in the second
paragraph above fails for all three groups when $k$ is large enough.
A standard fact from homological algebra then guarantees that 
\begin{multline*}
	\sum_{k=0}^\infty
	(-1)^k\rank(\MH_{k,l_j}(X))
	=
	\\
	\sum_{k=0}^\infty
	(-1)^k
	\left[\rank(\BMH_{k}(X)(l_j))
	-
	\rank(\BMH_{k}(X))(l_{j-1})
	\right].
\end{multline*}

For step 3, we have exchanged the order of summation.
This is valid because the series is absolutely convergent
(indeed, it is the same series as the one appearing in step 1).

For step 4, we have `telescoped' the sum, using the fact that
\[
	\rank(\BMH_k(X)(l_j))e^{-l_jt}\to 0
	\text{ as }
	j\to\infty.
\]
The latter holds because $\rank(\BMH_k(X)(l_j))$ is at most the number
of tuples $(x_0,\ldots,x_k)$ with consecutive entries distinct
and $\ell(x_0,\ldots,x_k)\le l_j$.  But then $l_j\ge k\delta$
so that $\rank(\BMH_{k}(X)(l_j))e^{-l_jt}\leq n^{k+1}e^{-l_jt}
\le n\cdot n^{l_j/\delta}e^{-l_j t}=n\cdot(n^{1/\delta}e^{-t})^{l_j}$.
But $(n^{1/\delta}e^{-t})<1$ by our standing assumption.
Since $l_j\to\infty$ as $j\to\infty$, the claim follows.

Step 5 follows by simply describing $\rank(\BMH_k(X)(l_j))$
as the number of bars in the barcode decomposition for
$\BMH_k(X)$ that contain $l_j$.
For step 6 we have again exchanged the order of summation,
which is valid because the series consists of non-negative numbers
and is convergent.
Step 7 is then a direct computation of the series
$\sum_{j=0}^\infty D(i,j,k)(e^{-l_jt}-e^{-l_{j+1}t})$.
\end{proof}

\section{Magnitude of persistence modules}
\label{section-persistent-magnitude}

In Theorem~\ref{theorem-magBMH} in the previous section, 
we saw a formula
expressing the magnitude function of a finite metric space $X$
in terms of the barcode decomposition of its blurred magnitude
homology.  In this section we will turn that result on its
head, and use the formula to \emph{define} a numerical invariant
of persistence modules and  graded persistence modules,
and explore its formal properties.

In subsequent sections we will apply our new invariant
to persistent homology groups, in order to obtain new
invariants of finite metric spaces (or of whatever input
the persistent homology theory accepts).

In this section we will usually work with finitely presented 
persistence modules, and finitely presented graded persistence modules.
In the latter case, we mean that the graded persistence module
has finitely many generators and relations in total, 
so that it is nonzero in only finitely many degrees.
Thus, our persistence modules will always be direct sums of finitely
many interval modules of the form $\kk[a,b)$ where possibly $b=\infty$.
This restriction allows us to work with the most relevant examples
such as Rips and \v Cech complexes while keeping technicalities to
a minimum.  At the end of the section we will offer two different
perspectives on the persistent magnitude, 
via the derived associated graded module
and the Laplace transform.  These offer
potential for extending the scope of persistent magnitude
beyond the present setting.

\subsection{Persistent magnitude}

\begin{definition}[Persistent magnitude]\label{definition-mag}
	Let $M$ be a finitely presented persistence module 
	with barcode decomposition 
	\[
		M
		\cong
		\bigoplus_{i=1}^n
		\kk[a_i,b_i).
	\]
	The \emph{persistent magnitude} 
	or simply \emph{magnitude} of $M$ is the real number
	\[
		|M|
		=
		\sum_{i=1}^n (e^{-a_i}- e^{-b_i})
	\]
	where by convention $e^{-\infty}=0$.
\end{definition}

\begin{example}\label{example-interval}
	For interval modules we have
	$|\kk[a,b)|=e^{-a}-e^{-b}$
	and
	$|\kk[a,\infty)| = e^{-a}$.
	Thus longer intervals have greater magnitude.
	This is in line with one of the general philosophies of persistent
	homology, that longer bars --- the features that persist longer ---
	are the more significant, while the shorter bars represent `noise'.
	(But note that there are other interpretations of barcodes, 
	especially of the shorter bars.
	For instance, it is shown
	in \cite{BHPWcurvature} that the short bars 
	in the barcode of points sampled from a disk of constant curvature
	can be used to infer the curvature.)
  Note also that interval modules of fixed length
	have greater magnitude the closer they are to $0$,
	i.e.~the sooner they begin.
\end{example}

\begin{proposition}[Additivity with respect to short exact sequences]
\label{proposition-ses}
	If
	\[
		0\to M\to N\to P\to 0
	\]
	is a short exact sequence
	of finitely presented persistence modules,
	then $|N| = |M| + |P|$.
\end{proposition}

We will give a proof of this proposition in 
section~\ref{subsection-associated-graded} below,
and another proof in section~\ref{subsection-laplace}.

\begin{definition}[Persistent magnitude of graded persistence modules]
	Let $M_\ast$ be a finitely presented graded persistence module.
	The \emph{persistent magnitude} of $M_\ast$ is defined as follows:
	\begin{align*}
		|M_\ast|
		&=
		\sum_{i} (-1)^i|M_i|
	\end{align*}
\end{definition}

If $C_\ast$ is a chain complex of persistence modules, 
then we obtain two graded persistence modules,
namely $C_\ast$ itself,
and the homology $H_\ast(C)$.
The persistent magnitude of these is related by the following
result, whose proof is a standard consequence of 
additivity with respect to short exact sequences.
(Compare with the proof of Theorem~2.44 of~\cite{Hatcher}.)

\begin{proposition}\label{proposition-maghom}
	Let $C_\ast$ be a finitely presented 
	chain complex of persistence modules.
	Then 
	\[
		|H_\ast(C)| = |C_\ast|.
	\]
\end{proposition}

\subsection{Rescaling and the magnitude function}

In Definition~\ref{definition-magnitude-function} the magnitude
of a finite metric space was extended from a number to a function
by means of rescaling the metric space.  We now do the same with
persistent magnitude.

\begin{definition}[Rescaling of persistence modules]
	Given a persistence module $M$ and a real number $t\in(0,\infty)$,
	we can define the \emph{rescaled module} $tM$ to be the new persistence
	module defined by
	\[
		tM (s) = M(s/t)
	\]
	for $s\in[0,\infty)$.
	More precisely, $tM$ is obtained from $M$ by precomposing
	with the functor from $[0,\infty)$ to itself
	that sends $s$ to $s/t$.
	This operation extends to graded persistence modules and chain
	complexes of persistence modules in the evident way.
\end{definition}

One can think of the definition of this rescaling operation as saying that 
features of $M$ that occur at $s$ become
features of $tM$ that occur at $ts$.

\begin{example}[Rescaling intervals and barcodes]
	One can check that $t\kk[a,b) = \kk[ta,tb)$.
	Thus the effect of the rescaling operation on the barcode
	of a finitely presented persistence module $M$
	is to simply rescale it by $t$: the barcode of $tM$ is obtained
	from that of $M$ by applying a scale factor of $t$ in the
	horizontal direction.
\end{example}

\begin{example}[Rescaling of metric spaces]
	Rescaling of persistence modules is designed to interact nicely
	with rescaling of metric spaces.  Recall that if $X$ is a metric
	space and $t\in (0,\infty)$, then the rescaling $tX$ is the metric space
	with the same underlying set and with metric defined by
	$d_{tX}(a,b) = t d_X(a,b)$.
	Then one can check directly that the Rips and \v Cech complexes
	satisfy
	\[
		tC_\ast(\calR(X)) = C_\ast\calR(tX)
		\quad\text{and}\quad
		tC_\ast(\calCh(X)) = C_\ast(\calCh(tX))
	\]
	and similarly for the persistent homology.
\end{example}

\begin{definition}[The persistent magnitude function]
	\label{definition-persistent-magnitude-function}
	The \emph{persistent magnitude function} 
	or simply \emph{magnitude function} of a finitely
	presented persistence module $M$ is the function 
	$(0,\infty)\to\R$ defined by
	\[
		t\longmapsto |tM|.
	\]
	If $M$ has direct sum decomposition 
	$M \cong \bigoplus_{i=1}^n \kk[a_i,b_i)$,
	then the magnitude function is given by the formula 
	\[
		|tM|
		=
		\sum_{i=1}^n (e^{-a_it}- e^{-b_it})
	\]
	where again by convention $e^{-\infty}=0$.
	\end{definition}

The extremal behaviour of the magnitude function
singles out two special classes of bars,
as we see in the next proposition.  Its proof is an
immediate consequence of the definitions.

\begin{proposition}\label{proposition-limiting}
	Let $M$ be a finitely presented persistence module.
	Then:
	\begin{itemize}
		\item
		$\lim_{t\to 0}|tM|$
		is the number of bars in $M$ of the form $\kk[a,\infty)$.
		\item
		if all bars of $M$ are contained in $[0,\infty)$ (or equivalently
		if $M(s)=0$ for $s<0$), then $\lim_{t\to\infty}|tM|$
		is the number of bars in $M$ of the form $\kk[0,b)$, 
			including the case $b=\infty$.
		\end{itemize}
\end{proposition}

One can think of this as follows:
As $t\to 0$, we are scaling down the barcode of $M$, so that any
finite bars eventually disappear at $0$, while any infinite bars
remain, but all become indistinguishable; in this limit
the magnitude function simply counts the latter.
As $t\to \infty$, we are scaling up the barcode of $M$,
so that any bars which begin \emph{after} $0$ eventually
disappear at infinity, while all bars that begin at $0$
remain but become indistinguishable; in this limit
the magnitude function again just counts the latter.

\begin{remark}[Reparameterisation of persistence modules]\label{remark-homeo}
Occasionally, it is useful to reparameterise a persistence module $M$ by an orientation preserving homeomorphism $h\colon\R\to\R$. In this case, we can define the {\em reparameterised module} $hM$ by:
\[
hM(s)=M(h^{-1}(s))
\]
The properties of this definition generalise the ones for rescaling by a positive real number in a natural way. For instance,
\[
h\kk[a,b)=\kk[h(a),h(b))
\]
and if $M\cong\bigoplus_{i=1}^n\kk[a_i,b_i)$ we have
\[
|hM|=\sum_{i=1}^n (e^{-h(a_i)}- e^{-h(b_i)}).
\]
\end{remark}

This definition has the following immediate but useful consequence:

\begin{lemma}\label{lemma-homeo}
Suppose $M$ is a finitely presented persistence module with magnitude function
\[
|tM|=\sum_{i=1}^n \lambda_i e^{-r_i t}
\]
and $h\colon\R\to\R$ is an orientation preserving homeomorphism. Then the reparameterised module $hM$ has magnitude function
\[
|thM|=\sum_{i=1}^n \lambda_i e^{-h(r_i)t}.
\]
\end{lemma}
%
%\begin{proof}
%This follows from Remark \ref{remark-homeo} by expressing the two magnitude functions in terms of the interval decomposition of $M$ and using the fact that the functions $(t\mapsto e^{-r t})_{r\in\R}$ are linearly independent.
%\end{proof}

\begin{remark}[Connection with the Euler characteristic of a barcode]
	In section~6 of~\cite{bobrowski}, Bobrowski and Borman
	define the \emph{Euler characteristic} of a barcode
	with no bars of length $\infty$,
	or in other words of a finitely presented graded persistence module
	$M$ with the property that $M(t)=0$ for $t$ sufficiently large.
	The definition is given by
	\[
		\chi(M_*)=\sum_{i=1}^b(-1)^{|\beta_i|}(b_i-a_i)
	\]
	where $\beta_1,\ldots,\beta_b$ are the bars of $M_\ast$, 
	and $\beta_i=[a_i,b_i)$.
	They then describe a connection between the Euler integral 
	of a function and the Euler characteristic of the barcode
	of the persistent homology of that function.
	Observe that the Euler characteristic of a finitely
	presented persistence module is related to its magnitude
	as follows:
	\[
		\chi(M_*)=|M_\ast|'(0)
	\]
	Thus the magnitude of a persistence module encodes its Euler
	characteristic.
	(In order to form $|M_\ast|'(0)$ above we have extended the magnitude
	function $t\mapsto |tM_\ast|$ to $0$ in the evident way
	(see the formula in 
	Definition~\ref{definition-persistent-magnitude-function})
	and then taken a one-sided derivative.)
\end{remark}

\subsection{Persistent magnitude and products}

Now we will explore how the persistent magnitude
interacts with the tensor product of persistence modules,
which was described in Section~\ref{section-background-persistence}.

\begin{proposition}[Persistent magnitude respects products]
\label{proposition-tensor}
	Let $M$ and $N$ be finitely presented persistence modules.
	Then
	\begin{equation}
	\label{equation-products}
		|M|\cdot|N|
		=
		|M\otimes N|
		-|\tor_1(M,N)|.	
	\end{equation}
\end{proposition}

\begin{proof}
	It suffices to prove this when $M$ and $N$ are interval
	modules.  In this case both sides of the equation
	can be computed directly using the results stated
	in Section~\ref{section-background-persistence}.
\end{proof}

Let us explain why equation~\eqref{equation-products} states that
`persistent magnitude respects tensor products', since it may look a little odd
from that point of view.
Homological algebra tells us that to fully understand the
tensor product of $M$ and $N$, we must consider the 
graded object $\tor_\ast(M,N)$. 
(Serre's intersection formula in algebraic geometry is a good
example of this principle in action.)
Thus, the `true' statement that persistent magnitude
respects products would be
\[
	|M|\cdot |N|
	=
	|\tor_\ast(M,N)|
	=
	\sum_{i=0}^\infty(-1)^i |\tor_i(M,N)|.
\]
But this reduces to exactly the equation appearing in the proposition.

The usefulness of Proposition~\ref{proposition-tensor} is that it
leads to product preserving properties of the magnitude of
persistent homology theories.  Indeed, suppose given a theory
that assigns to objects $X$ a graded persistence module
$A_\ast(X)$, and suppose that objects $X$ and $Y$ 
can be equipped with a product $X\times Y$ 
in such a way that we have a K\"unneth theorem:
\[
	0\to A_\ast(X)\otimes A_\ast(Y)
	\to
	A_\ast(X\times Y)
	\to
	\tor_1(A_\ast(X),A_{\ast-1}(Y))
	\to
	0
\]
Examples (and non-examples)
of such K\"unneth theorems are discussed in~\cite{BubenikMilicevic},
\cite{PSS} and~\cite{CarlssonFilippenko}.
In such a setting, Proposition~\ref{proposition-tensor} can be used to prove
the identity
\[
	|A_\ast(X\times Y)|
	=
	|A_\ast(X)|\cdot|A_\ast(Y)|.
\]
Indeed, we have
\begin{align*}
	|A_\ast(X\times Y)|
	&=
	|A_\ast(X)\otimes A_\ast(Y)|
	+
	|\tor_1(A_\ast(X),A_{\ast-1}(Y))|
	\\
	&=
	|A_\ast(X)\otimes A_\ast(Y)|
	-
	|\tor_1(A_\ast(X),A_{\ast}(Y))|
	\\
	&=
	|A_\ast(X)|\cdot|A_\ast(Y)|+|\tor_1(A_\ast(X),A_\ast(Y))|
	- |\tor_1(A_\ast(X),A_\ast(Y))|
	\\
	&=
	|A_\ast(X)|\cdot |A_\ast(Y)|
\end{align*}
where the first equality comes from the short exact sequence
(Proposition~\ref{proposition-ses})
and the third comes from~\eqref{proposition-tensor}.

\subsection{Persistent magnitude via
derived associated graded modules}
\label{subsection-associated-graded}

In this subsection we will give a perspective on the definition
of persistent magnitude using the homological algebra of the
`associated graded' or `causal onset' functor.
This will allow us to give a proof of Proposition~\ref{proposition-ses}.
It also gives a potential avenue for extending the definition
of persistent magnitude beyond the case of finitely presented modules.

As described in section~\ref{section-background-persistence}
we denote by $\persmod$ the category of persistence modules
and by $\grmod$ the category of $\R$-graded modules.
These are abelian categories, and $\persmod$ has enough projectives
--- they are the interval modules $\kk[a,\infty)$ for $a\in\R$.
For $a\in\R$ we let $\kk_a$ denote the object of $\grmod$
consisting of $\kk$ in grading $a$ and $0$ in all other gradings.

\begin{definition}
	The \emph{associated graded functor}
	$\Gr\colon\persmod\to\grmod$ is defined by
	\[
		\Gr(M)(s) = \frac{M(s)}{\sum_{s'<s}\mathrm{im}(M(s')\to M(s))}
	\]
	for $s\in\R$.
	The functor $\Gr$ is right exact, and we denote its derived functors
	by $\Gr_i(M)$, $i\geq 0$, with $\Gr_0(M)=\Gr(M)$.
\end{definition}

The terminology above was chosen because 
if the persistence module $M$ is obtained from an $\R$-filtered
vector space in the evident way, then $\Gr(M)$ is nothing other
than the $\R$-graded vector space associated to $M$.
The functor $M\mapsto \Gr(M)_s$ appears
in~\cite{vongmasacarlsson}, where it is denoted by $\mathcal{O}_s$, 
and called the \emph{causal onset} functor.

\begin{example}\label{example-gr}
	For a free module $\kk[a,\infty)$ we have 
	$\Gr_0(\kk[a,\infty))=\Gr(\kk[a,\infty))=\kk_a$ and $\Gr_1(\kk[a,\infty))=0$.
	And for an interval module $\kk[a,b)$ with $a<b$ we have a free 
	resolution
	\[
		\cdots\to 0 \to \kk[b,\infty)\to \kk[a,\infty)\to \kk[a,b)\to 0
	\]
	so that $\Gr_0(\kk[a,b))=\kk_a$ and $\Gr_1(\kk[a,b))=\kk_b$.
	It follows that for finitely presented modules,
	$\Gr_i(M)=0$ for $i>1$.
\end{example}

\begin{definition}[Graded magnitude]
	The \emph{graded magnitude} of a finitely presented
	object of $\grmod$, i.e.~a module of the form
	$\bigoplus_{i=1}^n \kk_{a_i}$ %where $a_1,\ldots,a_n\in[0,\infty)$,
	is
	\[
		\left|\bigoplus_{i=1}^n \kk_{a_i}\right|
		=
		\sum_{i=1}^n e^{-a_i}.
	\]
\end{definition}

The graded magnitude function is clearly additive with respect to short 
exact sequences of finitely presented graded modules.
The computations in Example~\ref{example-gr} give the following.

\begin{lemma}\label{lemma-graded-persistent}
	The persistent and graded magnitude are related as follows.
	Let $M$ be a finitely presented persistence module.
	Then
	\[
		|M| = |\Gr_0(M)| - |\Gr_1(M)|.
	\]
\end{lemma}

We are now in a position to prove that magnitude is additive with 
respect to short exact sequences.

\begin{proof}[Proof of Proposition~\ref{proposition-ses}]
	Like any derived functors, the $\Gr_i$ convert a short exact sequence
	\[
		0\to M\to N\to P\to 0
	\]
	into a long exact sequence
	\[
		0\to \Gr_1(M)\to \Gr_1(N)\to \Gr_1(P)
		\to \Gr_0(M)\to \Gr_0(N)\to \Gr_0(P)\to 0,
	\]
	and the statement of the proposition amounts to the claim
	that the alternating sum
	of the magnitudes of the modules in this sequence is zero.
	But this is a standard consequence of additivity
	with respect to short exact sequences,
	which in the case of graded magnitude is immediate
	from the definitions.
	(Compare with the proof of Theorem~2.44 of~\cite{Hatcher}.)
\end{proof}

\begin{remark}
	The proof of Proposition~\ref{proposition-ses} shows that
	the magnitude of a persistence module $M$ depends only
	on its `derived associated graded' modules $\Gr_0(M)$ and $\Gr_1(M)$.
	In simpler terms, the magnitude depends not on the lengths
	of the bars in the barcode, but only on the collection of start
	and end points of bars in the barcode.  
	One may then argue that magnitude does not contain 
	any `persistent' information, only `graded' information.
	However, the same comment can apply to \emph{any}
	invariant of persistence modules that is additive with respect
	to short exact sequences, thanks to the short exact sequences
	\[
		0\to \kk[b,\infty)\to \kk[a,\infty)\to \kk[a,b)\to 0.
	\]
	From the point of view of graded modules,
	persistence modules and persistent homology theories 
	should perhaps then be regarded as an excellent source 
	of interesting examples.
\end{remark}

\subsection{Persistent magnitude and the Laplace transform}
\label{subsection-laplace}

Here is an alternative approach to the persistent magnitude
using the Laplace transform, that we believe is the `correct'
way to understand persistent magnitude.

Let $M$ be a finitely presented persistence module and let 
\[
	\rank(M)\colon\R\to\R
\]
be its associated rank function,
i.e.~$\rank(M)(s)=\rank(M(s))$ for $s\in\R$.
This is a step function, given by the sum of the indicator
functions of the bars of $M$.  It has a derivative in the distributional
sense, given by
\[
	\rank(M)' = \sum_{i=1}^n(\delta_{a_i}-\delta_{b_i})
\]
where $M\cong\bigoplus_{i=1}^n \kk[a_i,b_i)$,
and where $\delta_x$ denotes the Dirac
delta distribution supported at $x$.
Recall that the \emph{bidirectional Laplace transform} $\mathcal{L}\{f\}$
of a function or distribution $f$ is given by
\[
	\mathcal{L}\{f\}(t)=\int_{-\infty}^\infty f(s)e^{-st}ds
\]
for $t\in[0,\infty)$.
Then one can check directly that
\[
	|tM| = \mathcal{L}\{\rank'(M)\}(t).
\]
In particular, the right hand side can be used as an alternative definition of
magnitude function, whereas magnitude is recovered by evaluating at $1$.
For example, if $M=\kk[a,b)$ then $\rank(M)$ is the step function
$1_{[a,b)}$ and $\rank'(M) = \delta_a - \delta_b$ so that
\[
	\mathcal{L}\{\rank'(M)\}(t)
	=
	\int_{-\infty}^\infty (\delta_a(s)-\delta_b(s))e^{-st}ds
	=
	e^{-at}-e^{-bt}
	=
	|tM|.
\]

From this point of view, the additivity of persistent magnitude
with respect to short exact sequences (Proposition~\ref{proposition-ses})
is an immediate consequence of the fact that if
\[
	0\to M\to N\to P\to 0
\]
is a short exact sequence of persistence modules, then
$\rank(N) = \rank(M) + \rank(P)$.

In the case of a finitely presented graded persistence module $M_\ast$,
we can associate to it its Euler characteristic curve
(see \cite[Section 3.2]{turner2014persistent} and \cite{heiss2017streaming,
bobrowski2014topology,fasy2018challenges} for some related work),
i.e.~$\chi(M_*)(s)=\sum_{i=0}^{\infty}(-1)^i\rank(M_i(s))$ for $s\in\R$.
Then we have:
\[
|tM_\ast|=\sum_{i=0}^{\infty}(-1)^i|tM_i|
=\sum_{i=0}^{\infty}(-1)^i\mathcal{L}\{\rank'(M_i)\}(t)
=\mathcal{L}\{\chi'(M_i)\}(t).
\]
In other words, the magnitude of a finitely presented graded
persistence module is precisely the Laplace transform of the
derivative of its Euler characteristic curve. This provides
yet another connection between magnitude and TDA.

More explicitly, whenever $M_*$ is a finitely presented graded
persistence module and $r_1<\ldots<r_n<r_{n+1}=\infty$ is the
sequence of all the startpoints and endpoints in its interval
decomposition, we have
\begin{equation}\label{magnitudeviaeuler}
|tM_*|=\sum_{j=1}^n \chi(M_*)(r_j)(e^{-r_j t}-e^{-r_{j+1} t}).
\end{equation}
One interpretation of this formula is that persistent magnitude
of a graded persistence module can be regarded as the `filtered
Euler characteristic' associated to it.
%(One obtains this description using the tools of Proposition~\ref{proposition-ses}
%and the fact that the chains of $\calR(X)$ is a free persistence module.)

We include this perspective as a useful alternative point of view,
as well as a potential avenue for generalising the magnitude from
finitely presented modules to more general modules that, despite
not being finitely presented, may nevertheless have a `rank function'
or `rank distribution' that we can then differentiate and subject to
the Laplace transform.
(Here we recall~\cite{chazal2016structure}, where persistence modules
that do not admit a barcode decomposition are nevertheless 
equipped with a persistence diagram.)

\begin{remark}[Multi-parameter persistent magnitude]
	The theory of persistent homology captures the topology 
	of families governed by a single parameter, 
	but one often encounters richer structures 
	that are governed by multiple parameters.
	This is the setting of multi-parameter 
	persistence~\cite{multiparameter}.

	An \emph{$r$-parameter persistence module} 
	is a functor $M\colon (\R^r,\leq)\to\Vect$, 
	where $(x_1,\ldots,x_r)\leq (y_1,\ldots,y_r)$ 
	if $x_i\leq y_i$ for all $i$.
	Thus a persistence module is a $1$-parameter persistence module 
	in the sense of this definition.

	Adapting the methodology and results of $1$-parameter persistence to the
	multi-parameter case is a difficult and ongoing problem.
	For example, it is shown in~\cite{multiparameter} 
	that there is no discrete invariant of 
	$r$-parameter persistence modules analogous to the 
	barcode or persistence diagram of a $1$-parameter persistence module.

	It is therefore interesting to wonder whether an 
	$r$-parameter generalisation of persistent magnitude is possible.
	A first guess may be that the $r$-parameter persistent 
	magnitude function of an $r$-parameter persistence module $M$ 
	is given by a multi-parameter version of the interpretation given in 
	this section, so that we have a (partially defined) 
	function from $(0,\infty)^r$ to $\R$ defined by
	\begin{equation}\label{equation-multi}
		(t_1,\ldots,t_r)\mapsto
		\int_{\R^r}\frac{\partial^r\rank(M)}
		{\partial s_1\cdots\partial s_r}(s_1,\ldots,s_r)
		e^{-(s_1t_1+\cdots+s_rt_r)}ds_1\cdots ds_r
	\end{equation}
	where $\rank(M)$ is the function $\R^r\to\R$ 
	giving the rank of $M$ at each point, 
	and the partial derivative is again taken in 
	an appropriate distributional sense.

	For example, if we take a ``half-open box'' 
	$[a_1,b_1)\times\cdots\times[a_r,b_r)$ in $\R^r$, 
	then we may define an $r$-parameter persistence module
	$\kk\bigl([a_1,b_1)\times\cdots\times[a_r,b_r)\bigr)$ 
	whose value is $\kk$ on points within the box and $0$ elsewhere,
	and whose associated morphisms are the identity maps where possible, 
	and the zero map otherwise.
	This is the $r$-parameter analogue of an interval module $\kk[a,b)$.
	In this case, the rank function of our module is the product of the 
	indicator functions of the intervals $[a_i,b_i)$, and then one may
	compute the integral in~\eqref{equation-multi} using Fubini's theorem
	to show that 
	\[
		|\kk\bigl([a_1,b_1)\times\cdots\times[a_r,b_r)\bigr)|
		=
		|\kk[a_1,b_1)|\cdots |\kk[a_r,b_r)|
	\]
	More generally, one can take $r_1$- and $r_2$-parameter persistence
	modules and take their `exterior tensor product' to obtain an
	$(r_1+r_2)$-parameter persistence module, and one can show that
	the persistent magnitude of this exterior tensor product is the
	product of the persistent magnitudes of the two factors.

	We would expect the definition~\eqref{equation-multi}
	to have other good formal properties,
	but the challenge would be to make it applicable, 
	in particular to ensure that it can be computed for the sort
	of multi-parameter persistence modules that arise in applications.
\end{remark}

\section{Magnitude and persistent homology of sublevel sets}
\label{section-sublevel}

An important class of filtrations that can be studied by methods of persistent homology are sublevel set filtrations;
the study of these is to a large extent inspired by Morse theory.
Recall that if $f\colon X\to\R$ is a 
continuous function, then the \emph{sublevel set persistent homology}
of $(X,f)$ is the graded persistence module defined by
$s\mapsto H_\ast(f^{-1}(-\infty,s])$.

Consider the case of a Morse function $f\colon M\to\R$ on a closed smooth manifold. Being Morse, it only has finitely many nondegenerate critical points. 
%The function $f$ is called {\em nonresonant} \cite{nicolaescu2007invitation} if the critical values associated with these critical points are pairwise distinct.
The magnitude of the sublevel set persistence module associated to 
$(M,f)$,
which we refer to as the {\em Morse magnitude of $(M,f)$}
and write $|t(M,f)|_{\mathrm{Morse}}$,
has an explicit formula in terms of the critical points.

\begin{theorem}
Let $f\colon M\to\R$ be a Morse function on a closed smooth manifold $M$,
Let $S\colon (\R,\leq)\to\Top$ be the sublevel set filtration 
given by $S(s) = f^{-1}(-\infty,s]$.
Then the magnitude function of the sublevel set persistent homology 
$H_\ast S\colon(\R,\leq)\to\GrVect$ is expressed as follows:
\[
|t(M,f)|_{\mathrm{Morse}}=|t(H_\ast S)|=\sum_{p} (-1)^{\ind(p)}e^{-f(p)t}
\]
where the sum is over all critical points of $f$.
\end{theorem}

\begin{proof}
A basic result of Morse theory
\cite[Theorem 3.1]{milnor2016morse}
states that if $a<b$ are real numbers such that $f^{-1}(a,b]$ 
contains no critical points of $f$, then
$M^b=f^{-1}(-\infty,b]$ deformation retracts onto $M^a=f^{-1}(-\infty,a]$.
It follows that the critical values (i.e.~the values $f(p)$ of $f$ at the
critical points $p$) are the startpoints and endpoints
of the interval decomposition of $H_\ast S$.
List the critical values as $v_1<v_2<\cdots<v_k$.
We may now use the description of magnitude as the filtered Euler characteristic \eqref{magnitudeviaeuler}:
\[
|t(H_\ast S)|=\sum_{i=1}^k\chi(M^{v_i})(e^{-v_it}-e^{-v_{i+1}t}),
\]
where $v_{k+1}$ is interpreted as $\infty$. 
Another basic result of Morse theory
\cite[Theorem 3.2, Remark 3.3 \& Remark 3.4]{milnor2016morse}
states the following.  Suppose that $b$ is a critical
value of $f$, and $a<b$ is such that there are no critical
values of $f$ in $(a,b)$, and let $p_1,\ldots,p_r$ be the critical
points of $f$ with critical value $b$.
Then $M^b$ has a subspace of the form 
$M^a\cup e^{\ind(p_1)}\cup\cdots\cup e^{\ind(p_r)}$, 
and $M^b$ deformation retracts onto 
$M^a\cup e^{\ind(p_1)}\cup\cdots\cup e^{\ind(p_r)}$.
Using this result,
we then have 
$\chi(M^b) = \chi(M^a) + \sum_{j=1}^r (-1)^{\ind(p_j)}
= \chi(M^a) + \sum_{p\colon\ind(p)=b} (-1)^{\ind(p)}$.
It follows that if $v$ is a critical value of $f$, then
\[
	\chi(M^v) = \sum_{p\colon f(p)\leq v} (-1)^{\ind(p)}
\]
where the sum is over critical points with critical value at most $v$.
We now have
\begin{align*}
	|t(H_\ast S)|
	&=
	\sum_{i=1}^k\chi(M^{v_i})(e^{-v_it}-e^{-v_{i+1}t})
	\\
	&= 
	\sum_{i=1}^k \sum_{p\colon f(p)\leq v_i}(-1)^{\ind(p)}
	(e^{-v_i t}-e^{-v_{i+1}t})
	\\
	&= 
	\sum_p 
	(-1)^{\ind(p)}
	\sum_{v_j\colon f(p)\leq v_j}
	(e^{-v_j t}-e^{-v_{j+1}t})
	\\
	&= 
	\sum_p (-1)^{\ind(p)} e^{-f(p) t}
\end{align*}
\end{proof}

\begin{remark}
This could be generalised in a straightforward way to the case of the sublevel set filtration associated to any tame function $f\colon X\to\R$ on a topological space $X$ using the concept of {\em homological critical value \cite{bubenik2014categorification,cohen2007stability,govc2016definition}}.
\end{remark}

\begin{example}[Distance filtration]
Consider a subset $A\subseteq \R^n$ and filter $\R^n$ by $B(A,r)=\bigcup_{a\in A}B(x,r)$. This is the sublevel set filtration associated to the distance function $x\mapsto d(x,A)$. Applying singular homology $H_*$ to this filtration yields a graded persistence module.

For example, for the standard embedding $i\colon S^{n-1}\hookrightarrow\R^n$ we obtain the persistence module consisting of a $\kk[0,\infty)$ bar in degree $0$ and a $\kk[0,1)$ bar in degree $n-1$. In particular, the associated magnitude function, which could also be called `the distance magnitude function' is $|t S^{n-1}|_{\mathrm{dist}}=1+(-1)^{(n-1)}(1-e^{-t})$.

In the case $A$ is finite, the graded persistence module obtained is isomorphic to the \v Cech persistent homology module associated to $A$. The corresponding magnitude function could therefore reasonably be called the `\v Cech magnitude function' of $A$ and denoted $|t A|_{\mathrm{\check{C}ech}}$.
\end{example}

\section{Rips magnitude}
\label{section-rips}

In this section we will apply the persistent magnitude
developed earlier to the persistent homology of the Rips 
complex in order to obtain a new, variant form of magnitude
of a finite metric space.  Here  we explore the basic properties
of this new invariant, before going into further detail in later sections.

\begin{definition}[Rips magnitude]
	Let $X$ be a finite metric space.  
	Then the chains $C_\ast(\calR(X))$ of its Rips complex
	are a chain complex of finitely presented persistence
	modules, concentrated in degrees less than 
	the cardinality of $X$. The same therefore holds for
	the homology $H_\ast(\calR(X))$.
	The \emph{Rips magnitude} of $X$ is defined to
	be the magnitude of the chains of the Rips complex 
	or equivalently the magnitude of its homology:
	\[
		|X|_{\mathrm{Rips}} 
		=|C_\ast(\calR(X))|
		= |H_\ast(\calR(X))|
	\]
	The \emph{Rips magnitude function} of $X$ is defined
	as $t\mapsto|tX|_{\mathrm{Rips}}$, which is equal to
	\[
		|tX|_{\mathrm{Rips}} 
		=|tC_\ast(\calR(X))|
		= |tH_\ast(\calR(X))|
	\]	
\end{definition}

\begin{proposition}
	The Rips magnitude (function) of a finite metric space $X$
	has the following properties.
	\begin{enumerate}
		\item
		The Rips magnitude is computed by the formula:
		\begin{equation}
		\label{subsets-formula}
			|tX|_{\mathrm{Rips}} = 
			\sum_{A\subseteq X,\, A\neq\emptyset}
			(-1)^{\#A-1}e^{-\diam(A)t}
		\end{equation}

		\item
		If $H_\ast(\calR(X))$ has barcode with
		bars $[a_{k,0},b_{k,0}),\ldots,[a_{k,m_k},b_{k,m_k})$
		in degree $k$, then:
		\begin{equation}
		\label{barcode-formula}
			|tX|_{\mathrm{Rips}} 
			= 
			\sum_{k=0}^{\#X-1}\sum_{j=0}^{m_k}
			(-1)^k(e^{-a_{k,j}t}-e^{-b_{k,j}t})
		\end{equation}
		
		\item
		If $0=d_0<d_1<\ldots<d_n$ is the set of all pairwise distances between
		elements of $X$ arranged in a sequence and $d_{n+1}=\infty$, then:
		\begin{equation}
		\label{euler-formula}
		|tX|_{\mathrm{Rips}}=\sum_{j=0}^n \chi(\calR_{d_j}(X))(e^{-d_j t}-e^{-d_{j+1} t}).
		\end{equation}

		\item
		$\lim_{t\to 0}|tX|_{\mathrm{Rips}} = 1$ 
		and 
		$\lim_{t\to\infty}|tX|_{\mathrm{Rips}}=|X|$.
	\end{enumerate}
\end{proposition}

The fourth part of the proposition suggests
that the Rips magnitude is an `effective number of points',
in the same spirit as the magnitude.

\begin{proof}
	For the first part, we use the description
	$|tX|_{\mathrm{Rips}} = |t C_\ast(\calR(X))|$.
	Now $C_\ast(\calR(X))$ has barcode with one
	bar for each nonempty subset $A$ of $X$,
	and this bar lies in degree $\#A-1$,
	and has type $[\mathrm{diam}(A),\infty)$.
	The definition of persistent magnitude gives 
	the result immediately.
	The second part follows from the 
	description $|tX|_{\mathrm{Rips}}=|t H_\ast(\calR(X))|$.
	The third part follows from either the first or the
	second part using the formula \eqref{magnitudeviaeuler}.
	The fourth part follows directly from
	the first, and it can also be deduced from the 
	barcode description given there using
	Proposition~\ref{proposition-limiting}.	
\end{proof}

\begin{example}[Rips magnitude of the one-point space]
	Let $X$ denote the space consisting of a single point $x$.
	Then it has precisely one nonempty subset, namely
	$X$ itself, and $\#X=1$ while $\diam(X)=0$.
	Using formula~\eqref{barcode-formula}
	then gives us $|tX|_{\mathrm{Rips}}=1$.
\end{example}

\begin{example}[Rips magnitude of two-point spaces]	
	Let $X=\{x_1,x_2\}$ be the two-point space in which
	$d_X(x_1,x_2)=d$ for some $d>0$.
	
	Let us compute $|tX|_{\mathrm{Rips}}$ using 
	\eqref{subsets-formula}.
	The nonempty subsets of $X$ are $A_1=\{x_1\}$,
	$A_2=\{x_2\}$ and $A_3=X$, with
	$\#A_1=1$, $\#A_2=1$, $\#A_3=2$,
	$\diam(A_1)=0$, $\diam(A_2)=0$ and $\diam(A_3)=d$.
	Thus we have
	\begin{align*}
		|tX|_{\mathrm{Rips}}
		&=
		(-1)^{1-1}\cdot e^{-0t}
		+
		(-1)^{1-1}\cdot e^{-0t}
		+
		(-1)^{2-1}\cdot e^{-dt}
		\\
		&=
		1 + 1 -e^{-dt}
		\\
		&=
		2-e^{-dt}.
	\end{align*}

	Let us also compute $|tX|_{\mathrm{Rips}}$ using~\eqref{barcode-formula}.
	The Rips-homology $H_\ast(\calR(X))$ has barcode
	with bars $[0,\infty)$ and $[0,d)$ in degree $0$,
	and no other bars.  Thus~\eqref{barcode-formula}
	gives us
	\begin{align*}
		|tX|_{\mathrm{Rips}}
		&=
		(e^{-0t}-e^{-\infty t})
		+
		(e^{-0t} - e^{-dt})
		\\
		&=
		(1-0)
		+
		(1 - e^{-dt})
		\\
		&=
		2-e^{-dt}.
	\end{align*}
	(Recall our convention that $e^{-\infty}=0$.)
\end{example}

\begin{example}
	The Rips magnitude is not necessarily increasing or convex, 
	it can attain negative values, and it can attain values
	greater than the cardinality of the space.
	For instance, for the complete bipartite graph $K_{5,6}$ (with
	the graph metric) the Rips magnitude is given by
	\[
	|t K_{5,6}|_{\mathrm{Rips}}=11-30e^{-t}+20e^{-2t}.
	\]
	with graph:
	\begin{center}
		\includegraphics[width=200pt]{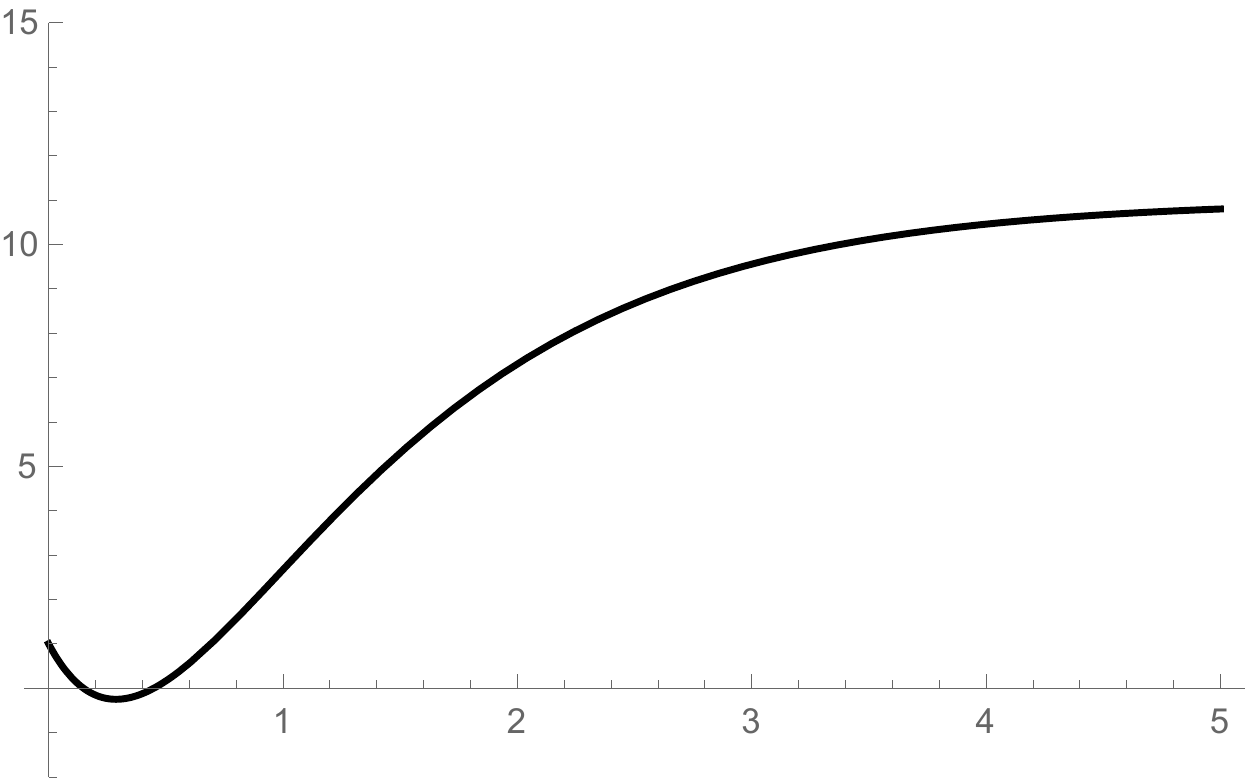}
	\end{center}
	And for the complete tripartite graph $X=K_{4,4,4}$, Rips magnitude
	is given by
	\[
	|t K_{4,4,4}|_{\mathrm{Rips}}=12+16e^{-t}-27e^{-2t}
	\]
	with graph:
	\begin{center}
		\includegraphics[width=200pt]{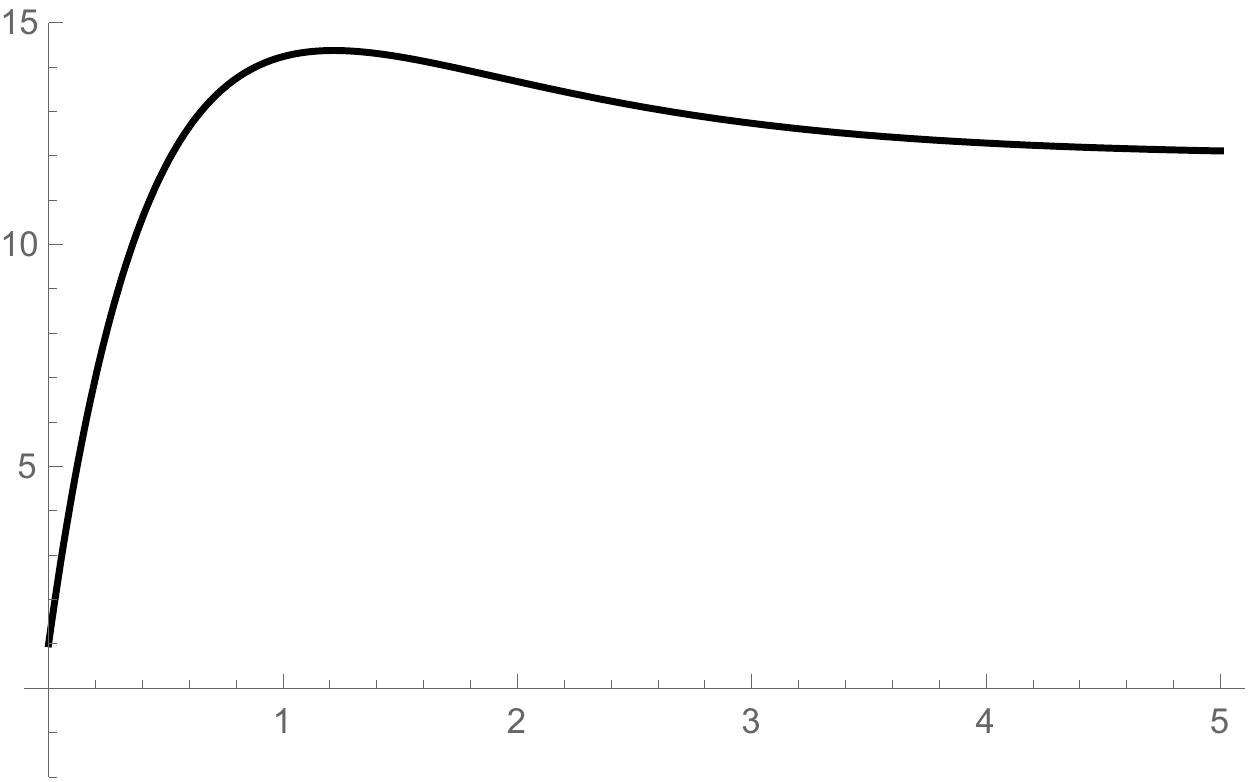}
	\end{center}
	Note that in general, if the metric only assumes integer values, 
	as in the case of a graph metric, 
	the associated Rips magnitude function is a polynomial in $q=e^{-t}$.
\end{example}

Now we consider the case of subsets of the real line,
where the computation is a little less trivial but accessible nonetheless.

\begin{proposition}\label{proposition-subsets-of-R}
	Let $A$ be a finite subset of the real line $\R$,
	with its induced metric.
	Order the elements of $A$ by size, $a_1<\ldots<a_n$. Then
	\[
	|t A|_{\mathrm{Rips}}=n-\sum_{j=1}^{n-1}e^{-(a_{j+1}-a_{j})t}.
	\]
\end{proposition}

We can think of this result as follows.
Take $A_1=\{a_1\}$, $A_2=\{a_1,a_2\}$, $A_3=\{a_1,a_2,a_3\}$
and so on, so that $A=A_n$.  Then the proposition tells us that
$|t A_1|_{\mathrm{Rips}}=1$, $|t A_2|_{\mathrm{Rips}}
= |t A_1|_{\mathrm{Rips}}+ (1-e^{-(a_2-a_1)t})$,
$|t A_3|_{\mathrm{Rips}} = |t A_2|_{\mathrm{Rips}} + (1-e^{-(a_3-a_2)t})$,
and so on. In other words, adding a point at the end increases the Rips
magnitude by $1-e^{-dt}$, where $d$ is the distance of the new
end point from the old one.  So if $d$ is very large, we increase
the Rips magnitude by almost $1$, whereas if $d$ is very small, then
we increase the Rips magnitude only a tiny amount.

\begin{proof}
	Given $B\subseteq A$, let $B_\mathrm{max}$ and $B_\mathrm{min}$ denote
	the maximum and minimum elements of $B$ respectively.
	Given $a\leq a'$ in $A$, let 
	$\mathcal{B}_{a,a'}$ denote the set of $B\subseteq A$
	for which $B_\mathrm{min}=a$ and $B_\mathrm{max}=a'$,
	and note that $\diam(B)=a'-a$ for all $B\in\mathcal{B}_{a,a'}$.
	Thus equation~\eqref{subsets-formula} gives us
	\[
		|t A|_{\mathrm{Rips}} = 
		\sum_{a\leq a'}
		e^{-(a'-a)t}
		\sum_{B\in\mathcal{B}_{a,a'}}
		(-1)^{\#B-1}.
	\]
	Now note the following:
	\begin{itemize}
		\item
		If $a=a'$, then $\mathcal{B}_{a,a'}$ consists
		of $\{a\}$ alone and 
		$\sum_{B\in\mathcal{B}_{a,a'}}(-1)^{\#B-1}=1$.

		\item
		If $a$ and $a'$ are adjacent elements of $A$,
		then $\mathcal{B}_{a,a'}$ consists of $\{a,a'\}$ alone
		and $\sum_{B\in\mathcal{B}_{a,a'}}(-1)^{\#B-1}=-1$.

		\item
		If $a$ and $a'$ are non-adjacent elements of $A$,
		then let $A_{a,a'}$ denote the set of those elements
		of $A$ that lie strictly between $a$ and $a'$.
		Then any $B\in\mathcal{B}_{a,a'}$ is the disjoint union
		of $\{a,a'\}$ with a subset $C\subseteq A_{a,a'}$.
		Thus
		$\sum_{B\in\mathcal{B}_{a,a'}}
		(-1)^{\#B-1}
		=
		\sum_{C\subseteq A_{a,a'}}(-1)^{\#C+1}
		=
		-\sum_{C\subseteq A_{a,a'}}(-1)^{\#C}
		=0$.
	\end{itemize}
	We therefore have
	\begin{align*}
		|t A|_{\mathrm{Rips}} 
		&= 
		\sum_{a}
		e^{-(a-a)t}\cdot 1
		+
		\sum_{\substack{a<a'\\ \text{adjacent}}}
		e^{-(a'-a)t}\cdot (-1)
		+	
		\sum_{\substack{a<a'\\ \text{non-adjacent}}}
		e^{-(a'-a)t}\cdot 0
		\\
		&= 
		n
		-
		\sum_{j=1}^{n-1}
		e^{-(a_{j+1}-a_j)t}
	\end{align*}
	as required.
\end{proof}

\section{Rips magnitude of cycle graphs and Euclidean cycles}
\label{section-cycles}

Other than the cases treated in the previous section, the topology of Rips complexes seems to be understood at all scales only in the case of finite subsets of the circle.  See ~\cite{adamaszek2017vietoris} and~\cite{AAFPPJ}.
In the case of Riemannian manifolds, Rips complexes are well understood at small scales thanks to a result of Hausmann~\cite[Theorem~3.5]{Hausmann}.
For ellipses, the Rips complex has been studied 
at a range of length scales in~\cite{AARellipses}.
Recently, it has also been shown that Rips complexes can be understood as nerves of certain covers via Dowker duality~\cite{Virk}.

In this section we focus on the circle $S^1$, which is equipped either with the Euclidean metric obtained from the standard embedding into $\R^2$, and denoted $S^1_\mathrm{eucl}$, or the geodesic (arclength) metric of total length $2\pi$, and denoted $S^1_\mathrm{geo}$.  

We will examine the subsets of equally spaced points in these spaces, whose corresponding Rips filtrations are well understood \cite{adamaszek}. Let $C_n^\mathrm{eucl}$ be the subset of $n$ equidistant points in $S^1_\mathrm{eucl}$ and let $C_n^\mathrm{geo}$ be the subset of $n$ equidistant points in $S^1_{geo}$.

Both of these can be related to cycle graphs. Let $C_n$ be the set of vertices of the $n$-cycle graph, equipped with the graph metric, where two adjacent vertices are considered to be at a distance of $1$. Note that this can be described as the subset of $n$ equidistant points in a geodesic circle of total arclength $n$. The Rips filtration of $C_n$ was studied in \cite{adamaszek} and the Rips filtrations of $C_n^\mathrm{geo}$ and $C_n^\mathrm{eucl}$ are just reparameterised versions of it.

More precisely, extend the function $[0,2]\to[0,\pi]$ given by $r\mapsto2\arcsin\frac{r}{2}$ to a homeomorphism $\phi\colon \R\to\R$. Then, by elementary trigonometry, we have the following relations:
\begin{equation}\label{reparameterised}
\calR_r(C_n^\mathrm{geo})=\calR_{\frac{n}{2\pi} r}(C_n),\qquad\calR_r(C_n^\mathrm{eucl})=\calR_{\frac{n}{2\pi}\phi(r)}(C_n)
\end{equation}
and
\begin{equation}\label{euclvsgeo}
\quad\calR_r(C_n^\mathrm{eucl})=\calR_{\phi(r)}(C_n^\mathrm{geo}).
\end{equation}

We now state the main results of this section and then proceed to prove them.

\begin{proposition}\label{ripsmag_cycle}
Writing $q=e^{-t}$, we have:
\[
|t C_n|_{\mathrm{Rips}}=\sum_{\substack{\text{odd }r|n\\r\neq n}}\frac nr q^{\frac nr \frac{r-1}2}(1-q)+q^{\lfloor\frac n2\rfloor}.
\]
\end{proposition}

This is reminiscent of certain functions that appear in analytic number theory, the simplest of which is probably the sum of divisors function $\sigma_k(n)=\sum_{d|n}d^k$. The appearance of sums over divisors of integers is quite surprising to us and seems to suggest that the Rips magnitudes of $n$-cycle graphs might be intimately connected to number theory in some way. We feel this connection could be worthy of further study:

\begin{question}
What is the connection between Rips magnitude of cycles and various functions studied in analytic number theory?
\end{question}

Using Lemma \ref{lemma-homeo} and equations \eqref{reparameterised} and \eqref{euclvsgeo}, Proposition \ref{ripsmag_cycle} immediately allows us to infer the following two corollaries:

\begin{corollary}\label{ripsmag_eucl}
The Rips magnitude of Euclidean cycles is given by\footnote{In fact, $\delta_n=\delta_{n,n}$ for odd $n$, so the condition $r\neq n$ can be omitted for Euclidean cycles.}:
\[
|t C^{\mathrm{eucl}}_n|_{\mathrm{Rips}}=\sum_{\substack{\text{odd }r|n\\r\neq n}}\frac nr (e^{-\delta_r t}-e^{-\delta_{r,n}t})+e^{-\delta_n t},
\]
where
\[
\delta_r=\diam(C^{\mathrm{eucl}}_r)=2\sin\left(\pi\frac{\lfloor\frac r2\rfloor}r\right)\quad\text{and}\quad\delta_{r,n}=2\sin\left(\pi\left(\frac1n+\frac{\lfloor\frac r2\rfloor}r\right)\right).
\]
\end{corollary}

\begin{corollary}\label{ripsmag_geo}
The Rips magnitude of geodesic cycles is given by:
\[
|t C^{\mathrm{geo}}_n|_{\mathrm{Rips}}=\sum_{\substack{\text{odd }r|n\\r\neq n}}\frac nr (e^{-\eta_r t}-e^{-\eta_{r,n}t})+e^{-\eta_n t},
\]
where
\[
\eta_r=2\pi\frac{\lfloor\frac r2\rfloor}r\qquad\text{and}\qquad\eta_{r,n}=2\pi\left(\frac1n+\frac{\lfloor\frac r2\rfloor}r\right).
\]
\end{corollary}

\begin{example}
The graph of $|t C^{\mathrm{eucl}}_{60}|_{\mathrm{Rips}}$
\begin{center}
\includegraphics[width=200pt]{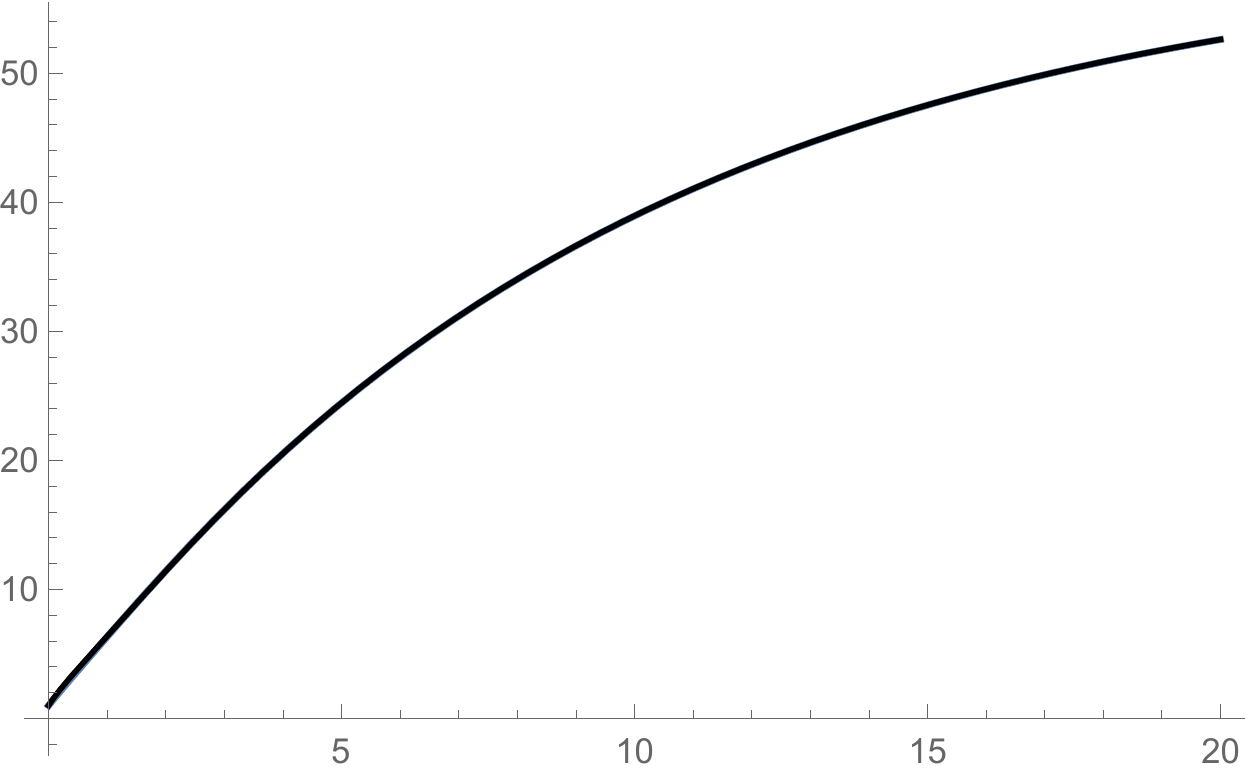}
\end{center}
looks deceptively similar to the one for $|t C^{\mathrm{geo}}_{60}|_{\mathrm{Rips}}$
\begin{center}
\includegraphics[width=200pt]{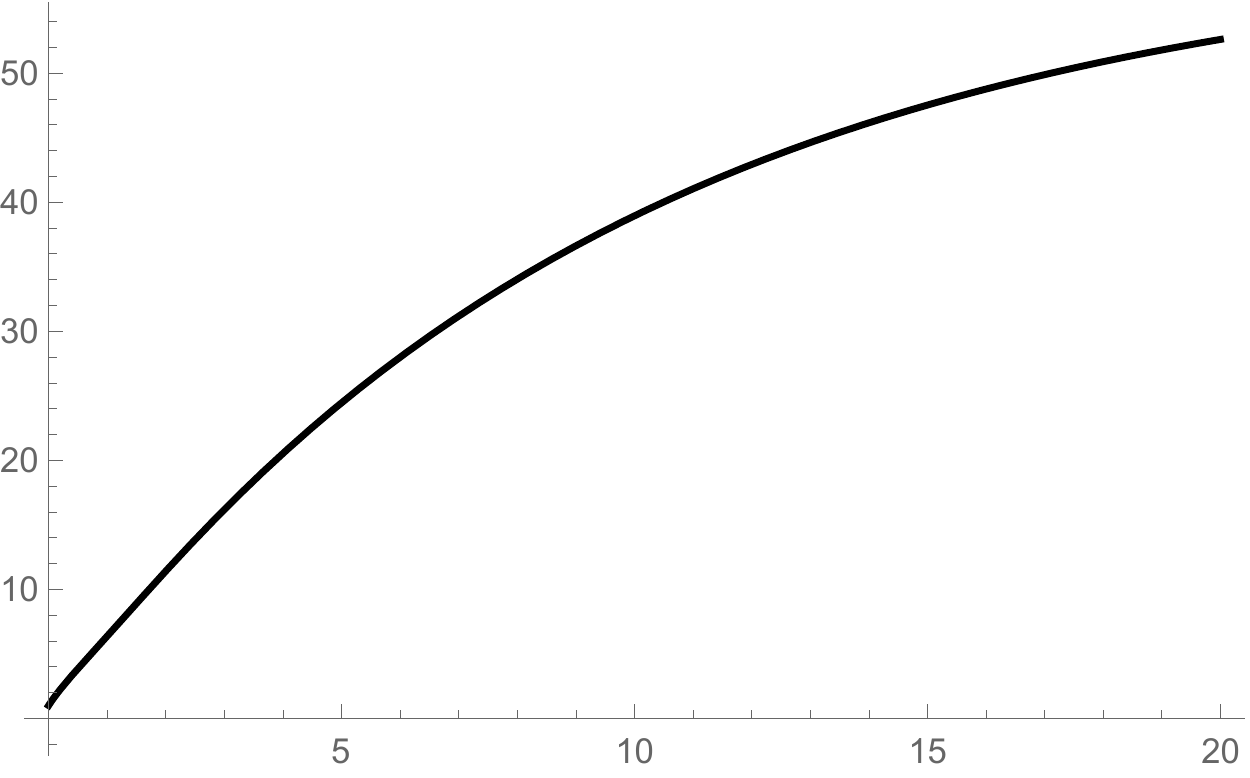}
\end{center}
so it is perhaps more instructive to look at the difference $|t C^{\mathrm{geo}}_{60}|_{\mathrm{Rips}}-|t C^{\mathrm{eucl}}_{60}|_{\mathrm{Rips}}$:
\begin{center}
\includegraphics[width=200pt]{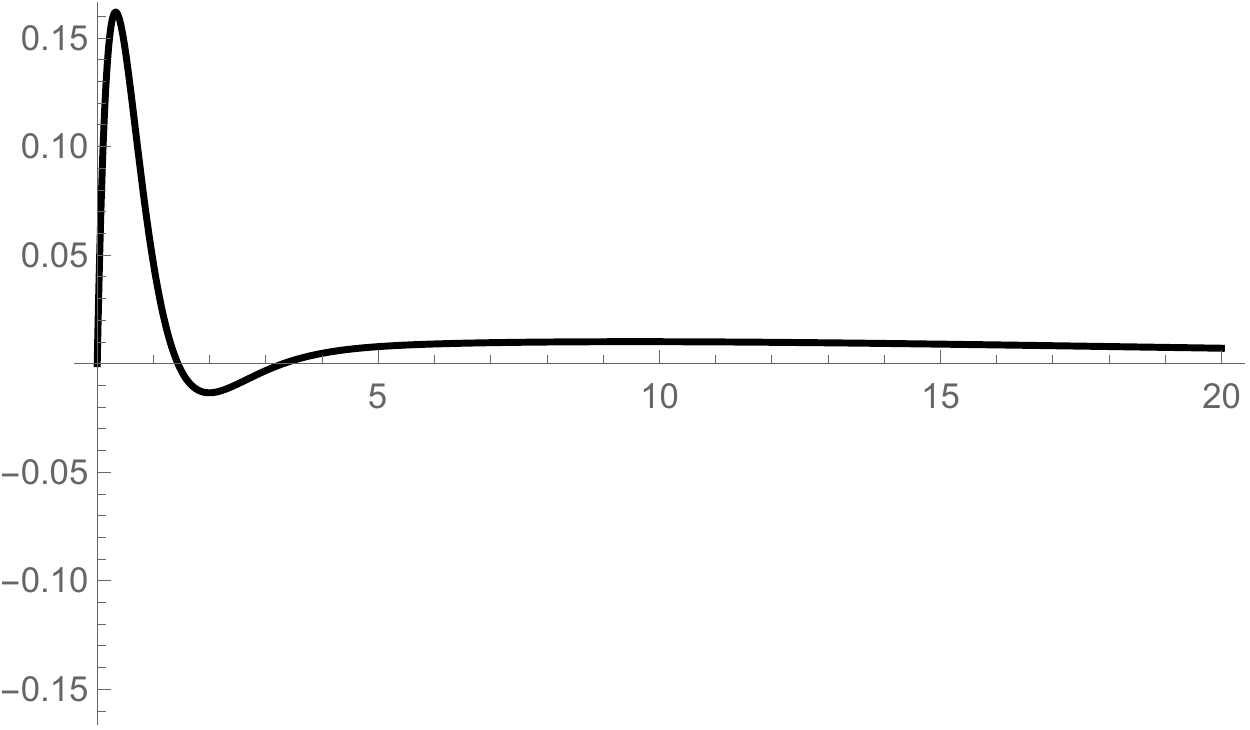}
\end{center}
\end{example}

\begin{proof}[Proof of Proposition \ref{ripsmag_cycle}]
Note that the Rips filtration of $C_n$ only has ``jumps'' at the integers. More precisely:
\[
\calR_r(C_n)=\calR_{\lfloor r\rfloor}(C_n).
\]
So it is sufficient to understand the Rips filtration at integer values of the filtration parameter $r$. For integer $r$ such that $0\leq r<\frac{n}2$, Adamaszek \cite[Corollary 6.7]{adamaszek} gives the following description of the homotopy types of the various stages of the Rips filtration:
\[
\calR_r(C_n)\simeq\begin{cases}\bigvee_{n-2r-1}S^{2l};&r=\frac l{2l+1}n,\\
S^{2l+1};&\frac l{2l+1}n<r<\frac{l+1}{2l+3}n.
\end{cases}
\]
From this we can immediately infer the Euler characteristics:
\[
\chi(\calR_r(C_n))=\begin{cases}n-2r;&\text{if $\frac{n}{n-2r}$ is an odd integer,}\\
1;&\text{if $n=2r$,}\\
0;&\text{otherwise.}
\end{cases}
\]
Using \eqref{euler-formula}, this implies
\begin{align*}
	|t C_n|_{\mathrm{Rips}}
	&=\sum_{r=0}^{\lfloor\frac{n}2\rfloor-1}\chi(\calR_r(X))(e^{-r t}-e^{-(r+1) t})+\chi(\calR_{\lfloor\frac{n}2\rfloor}(X))e^{-(\lfloor\frac{n}2\rfloor) t})
	\\
	&=\sum_{\substack{\text{odd }d|n\\d\neq n}}\frac nd q^{\frac nd \frac{d-1}2}(1-q)+q^{\lfloor\frac n2\rfloor}
\end{align*}
where the indices in the first and second summations are related by $\frac{n}{n-2r}=d$.
The claim follows.
\end{proof}

\begin{remark}
Another way to calculate the Euler characteristic of $\calR_r(C_n)$ would be from the simplex counts. We have computational evidence that the number of $i$-simplices in $\calR_r(C_n)$ for $r<\diam C_n$ is given by\footnote{Here, we use the \textbf{nonstandard (!)} convention that $\binom{n}{k}=0$ for $n<0$.}
\[
N_{n,r,i}=\sum_{k=0}^{\lfloor\frac n2\rfloor}\frac n{2k+1}\binom{(2k+1)r-kn+2k}{2k}\binom{(2k+1)r-kn}{i-2k}.
\]
\end{remark}

\begin{remark}
Computational evidence seems to suggest that the Rips magnitude $|t C_n|_{\mathrm{Rips}}$ of a cycle graph is convex if and only if
\[
n\in\{1, 2, 3, 4, 5, 6, 9, 10, 12, 15, 18, 21, 24, 27, 30, 33, 36, 39, 42, 48, 51, 57\}.
\]
\end{remark}

\section{Rips magnitude of infinite metric spaces}
\label{section-infinite}

We will now try to understand the Rips magnitude of infinite metric spaces.
In the original setting of magnitude,
positive definite spaces $X$, such as subspaces of Euclidean space,
have the property that if $A\subseteq B\subseteq X$ with $A,B$ finite,
then $|A|\leq |B|$.
Thus one definition of the
magnitude of an infinite metric space $X$ is
\[
	|X| = \sup_{\substack{A\subseteq X\\ A\text{ finite}}}|A|.
\]
Proceeding by analogy, we could attempt to make the definition
\[
	|X|_{\mathrm{Rips}} = \sup_{\substack{A\subseteq X\\ A\text{ finite}}}|A|_{\mathrm{Rips}}
\]
or the stronger definition
\[
	|X|_{\mathrm{Rips}} = \lim_{\substack{A\to X\\ A\text{ finite}}}|A|_{\mathrm{Rips}},
\]
where $A\to X$ in the Hausdorff metric.
The issues here are whether the supremum above exists, how one computes its
actual values, and the corresponding questions for the limit.

In this section we will explore the above situation in the case of the
unit interval, the circle with its Euclidean metric, and the circle
with its geodesic metric.  In the case of the interval the situation is as
good as one can hope for, with the conclusion that the Rips magnitude
of the interval `is' the function $t\mapsto 1+t$.
In the circle cases the situation is more ambiguous, and while the supremum
above may well exist (the supremum taken over all \emph{equally spaced}
subsets certainly does), the limit does not.

To improve readability, we state the results for each of the cases treated in a
separate subsection, while deferring all the proofs to one final subsection.

\subsection{The unit interval}

First  let us consider the unit interval.
Here we can give the following definitive description
of the situation.

\begin{theorem}\label{theorem-interval}
	Let $I=[0,1]$ denote the unit interval.
	\begin{enumerate}
		\item
		If $A,B\subseteq I$ are finite subsets with $A\subseteq B$,
		then
		\[
			|A|_{\mathrm{Rips}}\leq |B|_{\mathrm{Rips}}.
		\]

		\item
		For any $t\in (0,\infty)$ we have
		\[
			\sup_{\substack{A\subseteq I\\\text{finite}}}|t A|_{\mathrm{Rips}}
			= 1+t
		\]
		 and indeed
		\[
			\lim_{\substack{A\to I\\A\;\text{finite}}}|t A|_{\mathrm{Rips}}
			=
			1+t
		\]
		with uniform convergence on compact subsets of $(0,\infty)$.
	\end{enumerate}
\end{theorem}

Thus it seems reasonable in this situation to declare the Rips
magnitude of $I=[0,1]$ to be the function $|t[0,1]|_{\mathrm{Rips}}=1+t$.

\subsection{Euclidean Circle}

One possible approach to try and make sense of the notion of ``Rips magnitude of the Euclidean circle'' is by studying the behaviour of $|t C^{\mathrm{eucl}}_n|_{\mathrm{Rips}}$ as $n\to\infty$. 
As it turns out, however, this behaviour is not as straightforward as one might hope.

For instance, we show that $|t C^{\mathrm{eucl}}_n|_{\mathrm{Rips}}$ does not converge as $n\to\infty$, despite the fact that the Hausdorff distance $d_H(C^{\mathrm{eucl}}_n,S^1_{\mathrm{eucl}})=2\sin\left(\frac{\pi}{2n}\right)$ converges to $0$ as $n\to\infty$. We show this by first studying the behaviour of sequences of the form $|t C^{\mathrm{eucl}}_{mp}|_{\mathrm{Rips}}$, for fixed $m\in\N$ and where $p$ runs through all primes,
and showing that the limit along each such subsequence exists. However, these limits are different for different values of $m$.

We then show that despite this inconsistency, the $|t C^{\mathrm{eucl}}_n|_{\mathrm{Rips}}$ has finite upper and lower limits ($\liminf$ and $\limsup$) as $n\to\infty$ which can be expressed explicitly. These could be considered to be the `upper' and `lower Rips magnitude'. We also show that the `upper Rips magnitude' is equal to $\sup_{n\in\N}|t C^{\mathrm{eucl}}_n|_{\mathrm{Rips}}$. (Compare this with the first proposed definition of Rips magnitude at the beginning of this section.)

\begin{theorem}\label{theorem-EC-lim}
For any $m\in\N$:
\[
\lim_{\substack{p\to\infty\\p\text{ prime}}}|t C^{\mathrm{eucl}}_{mp}|_{\mathrm{Rips}}=e^{-2t}+2\pi t\sum_{\text{odd }r|m}\frac1r e^{-2t\cos\left(\frac{\pi}{2r}\right)}\sin\left(\frac{\pi}{2r}\right)
\]
\end{theorem}

This result means in particular that
\[
\lim_{n\to\infty}|t C^{\mathrm{eucl}}_n|_{\mathrm{Rips}}
\]
cannot exist. For instance, we have
\[
	\lim_{\substack{p\to\infty\\p\text{ prime}}}|t C^{\mathrm{eucl}}_p|_{\mathrm{Rips}}
	=e^{-2t}+2\pi t
\]
but
\[
	\lim_{\substack{p\to\infty\\p\text{ prime}}}|t C^{\mathrm{eucl}}_{3p}|_{\mathrm{Rips}}
	=
	e^{-2t}+2\pi t+\frac{\pi t}3 e^{-\sqrt{3}t}
\]
and the two limits do not coincide.

\begin{theorem}
\label{theorem-lim-sup-inf}
The Rips magnitudes of Euclidean cycles $C^{\mathrm{eucl}}_n$ satisfy
\[
\liminf_{n\to\infty}|t C^{\mathrm{eucl}}_n|_{\mathrm{Rips}}=e^{-2t}+2\pi t
\]
and
\[
\limsup_{n\to\infty}|t C^{\mathrm{eucl}}_n|_{\mathrm{Rips}}=e^{-2t}+2\pi t\sum_{r\text{ odd}}\frac1r e^{-2t\cos\left(\frac\pi{2r}\right)}\sin\left(\frac\pi{2r}\right).
\]
This series converges absolutely since its $r$-th term is bounded above by $\frac{\pi}{2r^2}$.
\end{theorem}

In fact, the upper limit is also the supremum of the sequence:

\begin{theorem}
\label{theorem-sup}
The Rips magnitudes of Euclidean cycles $C^{\mathrm{eucl}}_n$ satisfy
\[
\sup_{n\in\N}|t C^{\mathrm{eucl}}_n|_{\mathrm{Rips}}=e^{-2t}+2\pi t\sum_{r\text{ odd}}\frac1r e^{-2t\cos\left(\frac\pi{2r}\right)}\sin\left(\frac\pi{2r}\right).
\]
\end{theorem}

We include a plot of the $\liminf$ and $\limsup$ for small $t$ for comparison:
\begin{center}
\includegraphics[width=300pt]{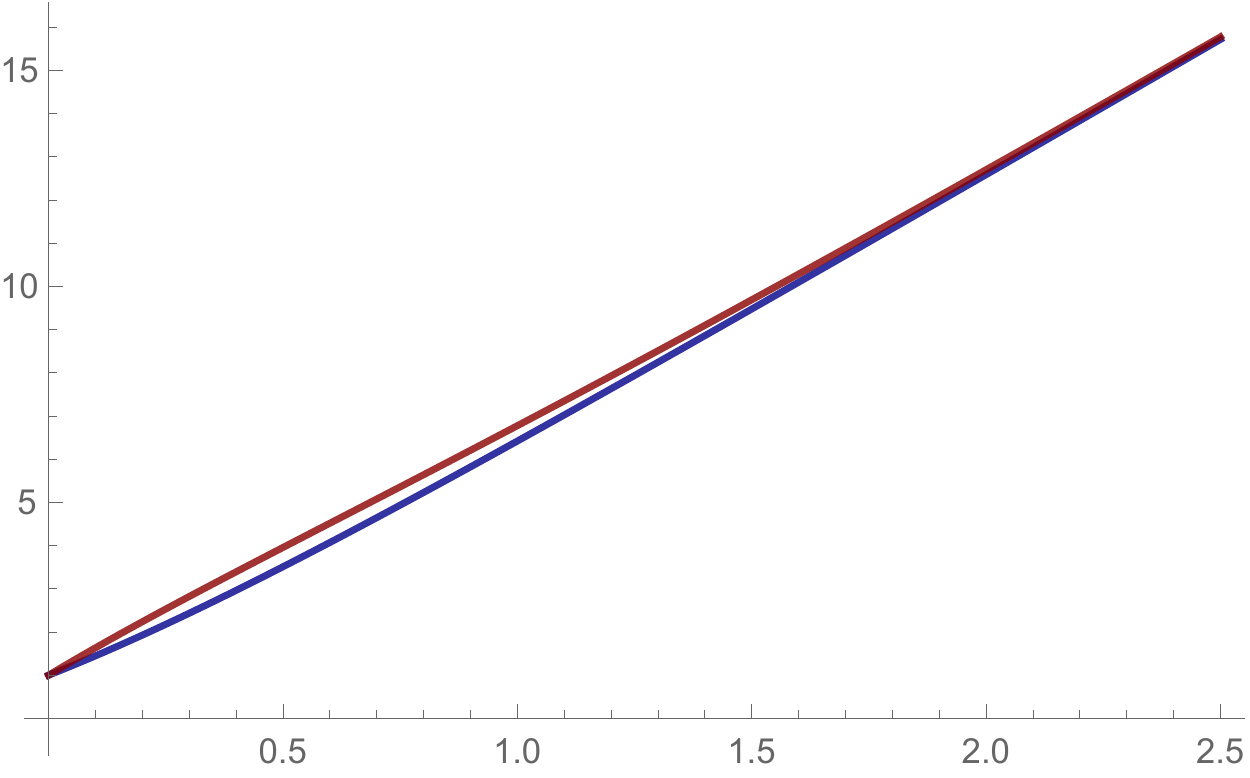}
\end{center}
It is visible from this graph that the difference is the most pronounced for small values of $t$. Now, fix $t=\frac12$. The behaviour of $|t C^{\mathrm{eucl}}_n|_{\mathrm{Rips}}$ evaluated at $t=\frac12$ as $n$ grows larger can be pictured as follows, with $n$ on the horizontal axis and $|t C^{\mathrm{eucl}}_n|_{\mathrm{Rips}}$
on the vertical.
The values of the $\liminf$ and $\limsup$ at $t=\frac12$ 
are plotted as the two red lines.

\begin{center}
\includegraphics[width=300pt]{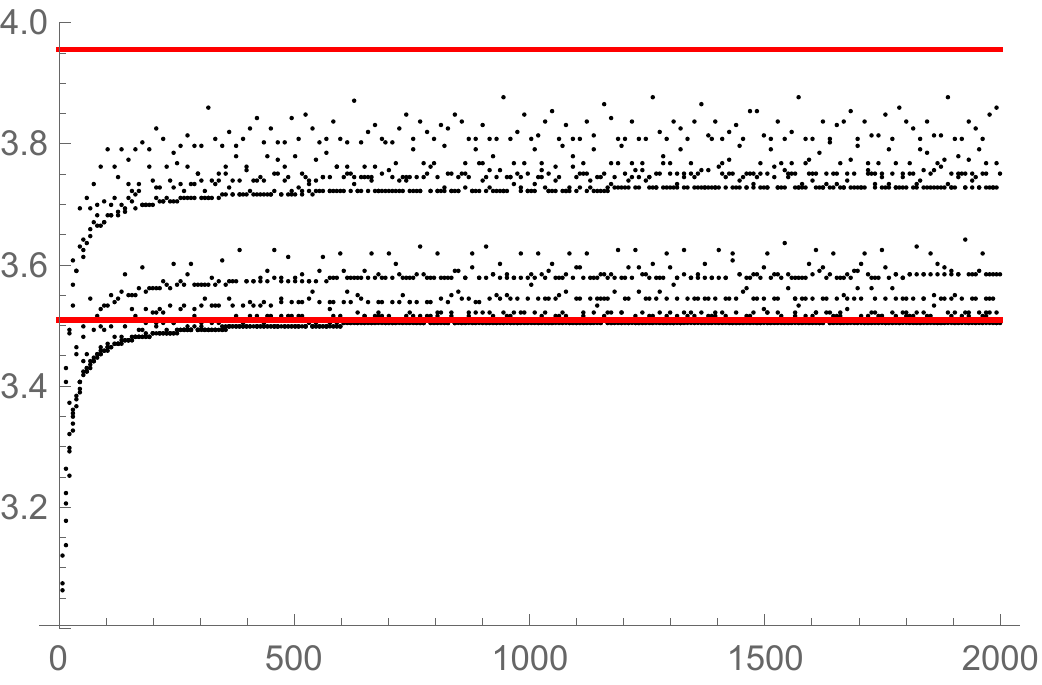}
\end{center}

Note that the chaotic behaviour of the graph is reminiscent of various functions
from analytic number theory such as for instance the sum of divisors function
$\sigma(n)=\sum_{d|n}d$, and is the graphical expression of the behaviour
discussed in the paragraph following Proposition \ref{ripsmag_cycle}.
It does appear as though the points accumulate more along specific lines,
which we suspect correspond to subsequences with certain divisibility properties.

Finally, note that restricting to equally spaced subsets of $S^1_{eucl}$ is somewhat
unnatural; we leave open the following question, which would define the `upper and
lower Rips magnitude' of the circle intrinsically:
\begin{question}
Does this asymptotic behaviour extend to arbitrary finite subsets of $S^1_{eucl}$?
For instance, given any $\epsilon>0$, is there a $\delta>0$ such that for all 
finite $A\subseteq S^1_{eucl}$ with $d_H(A,S^1_{eucl})<\delta$ we have
\[
e^{-2t}+2\pi t-\epsilon<|tA|_\mathrm{Rips}<e^{-2t}+2\pi t\sum_{r\text{ odd}}\frac1r e^{-2t\cos\left(\frac\pi{2r}\right)}\sin\left(\frac\pi{2r}\right)+\epsilon
\]
for all $t$ in a given interval?
\end{question}

\subsection{Geodesic circle}

Finally, we note that the case of the geodesic circle $S^1_{geo}$ of total arclength $2\pi$
can be treated using the same methods as we used for $S^1_{eucl}$. Namely, we restrict
attention to equally spaced subsets $C_n^{geo}$ described in Section \ref{section-cycles}.
(Note that $C_n^{geo}$ is just the $n$-cycle graph $C_n$ rescaled by $\frac{2\pi}{n}$.)

We could calculate the limits along the same subsequences we examined in the case $S^1_{eucl}$
and find that they again exist, but instead we just state the final result regarding the lower
and upper limit. In this case it turns out that the lower limit is still finite and can be expressed
explicitly, whereas the upper limit becomes infinite. Thus the Rips magnitude of $S^1_{eucl}$
and $S^1_{geo}$ behave quite differently.

\begin{theorem}\label{theorem-GC}
The Rips magnitudes of geodesic cycles $C^{\mathrm{geo}}_n$ satisfy
\[
\liminf_{n\to\infty}|t C^{\mathrm{geo}}_n|_{\mathrm{Rips}}=e^{-\pi t}+2\pi t
\]
and
\[
\limsup_{n\to\infty}|t C^{\mathrm{geo}}_n|_{\mathrm{Rips}}=\infty.
\]
\end{theorem}

The behaviour of $|t C^{\mathrm{geo}}_n|_{\mathrm{Rips}}$ evaluated at $t=\frac12$ as $n$ grows larger can be pictured as follows; the $\liminf$ evaluated at $t=\frac12$ corresponds to the red line and the $\limsup$ is infinite.

\begin{center}
\includegraphics[width=300pt]{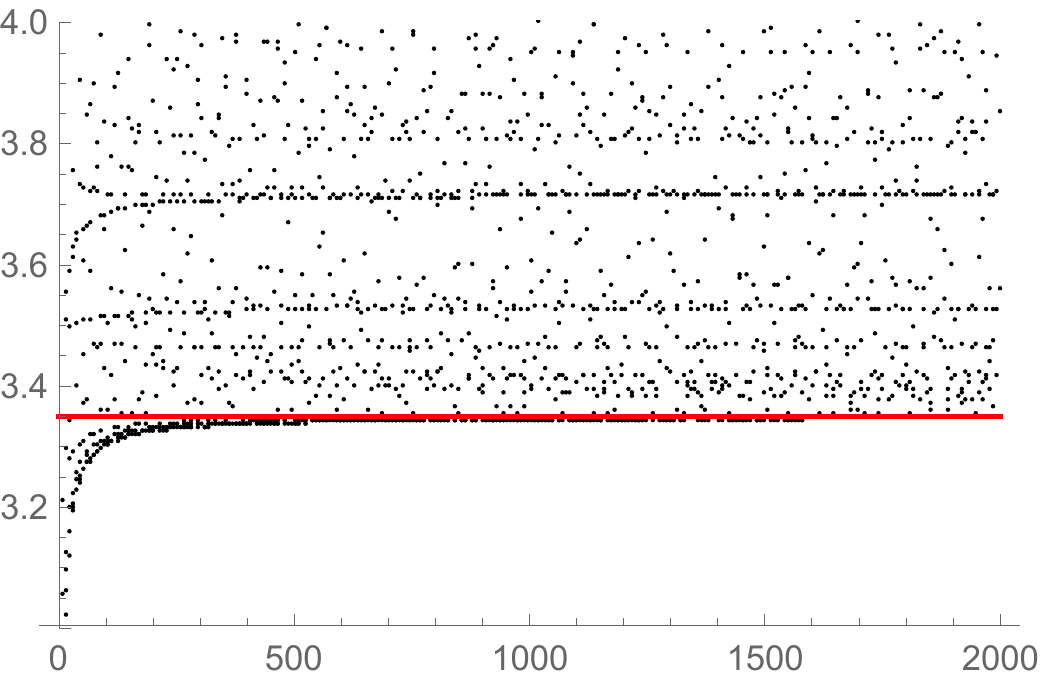}
\end{center}

We again note that it would be interesting to study the asymptotics over all finite subsets.

\subsection{Proofs of the Results}

Here we prove the results stated in the preceding subsections. We only give a sketch of the proof of Theorem \ref{theorem-GC} as the ideas are analogous to the ones used in the proof of Theorem \ref{theorem-lim-sup-inf}.

\begin{proof}[Proof of Theorem~\ref{theorem-interval} part (1)]
Proposition~\ref{proposition-subsets-of-R} shows that if $A\subseteq[0,1]$
is finite, with elements $a_1<\cdots<a_n$, then
\[
	|A|_{\mathrm{Rips}} = n - \sum_{j=1}^{n-1} e^{-(a_{j+1}-a_j)}.
\]
Suppose that $a_0<a_1$.  Then
\[
	|\{a_0\}\cup A|_{\mathrm{Rips}} 
	=
	|A|_{\mathrm{Rips}} + (1-e^{-(a_1-a_0)})
	\geq
	|A|_{\mathrm{Rips}},
\]
and similarly for $|A\cup\{a_{n+1}\}|_{\mathrm{Rips}}$ if $a_{n+1}>a_n$.
Now suppose that $a_i<b<a_{i+1}$.
Then
\begin{align*}
	|A\cup\{b\}|_{\mathrm{Rips}}
	&=
	|A|_{\mathrm{Rips}} + 1 + e^{-(a_{i+1}-a_i)}- e^{-(a_{i+1}-b)}-e^{-(b-a_i)}
	\geq |A|_{\mathrm{Rips}},
\end{align*}
the latter because one can easily see that 
$1+e^{-(x+y)}-e^{-x}-e^{-y}\geq 0$ for $x,y\geq 0$.
\end{proof}

\begin{proof}[Proof of Theorem~\ref{theorem-interval} part (2)]
Suppose $A$ consisting of $a_1<\ldots<a_n$ is a finite subset of the interval such that $d_H(A,I)<\delta<1$. Further define $a_0=0$ and $a_{n+1}=1$. The assumption on the Hausdorff distance implies that $a_{j+1}-a_j<\delta$ for $j=0,\ldots,n$ and $\delta<1$ implies $\delta^k\leq\delta$ for $k\in\N$. 
Consider the function $f\colon(0,\infty)\to\R$ defined by
\begin{align*}
	f(t)
	&=
	(1+t)-|t A|_{\mathrm{Rips}}
	\\
	&=
	(1+t)-n+\sum_{j=1}^{n-1} e^{-(a_{j+1}-a_j)t}
	\\
	&=t+\sum_{j=1}^{n-1}(e^{-(a_{j+1}-a_{j})t}-1).
\end{align*}
We extend the domain of $f$ to $[0,\infty)$ in the evident way.
We will show that:
\begin{enumerate}
	\item
	$f(0)=0$

	\item
	$f'(0)$ lies in the range
	$0\leq f'(0)\leq 2\delta$.

	\item
	$f''(t)$ lies in the range
	$0\leq f''(t)\leq\delta$
	for all $t\in[0,\infty)$.
\end{enumerate}
It follows quickly that
\[
	0\leq f(t)\leq \delta\cdot (2t+t^2/2)
\]
for $t\in[0,\infty)$.
Thus, as $\delta\to 0$ we have $f\to 0$ 
uniformly on any bounded subset of $[0,\infty)$, 
and the result follows.

It remains to check the three properties.  
The first and second derivatives of $f$ are as follows.
\begin{align*}
	f'(t)
	&= 
	1-\sum_{j=1}^{n-1}(a_{j+1}-a_j)e^{-(a_{j+1}-a_{j})t}
	\\
	f''(t)
	&= 
	\sum_{j=1}^{n-1}(a_{j+1}-a_j)^2e^{-(a_{j+1}-a_{j})t}
\end{align*}
Then (1) is immediate, while 
$f'(0)=1-\sum_{j=1}^{n-1}(a_{j+1}-a_j)=(1-a_n)+a_1$ and (2) follows,
and 
$0\leq f''(t)
\leq \sum_{j=1}^{n-1}\delta\cdot(a_{j+1}-a_j)\cdot 1$
and (3) follows.
\end{proof}

\begin{proof}[Proof of Theorem~\ref{theorem-EC-lim}]
Recall that by Corollary \ref{ripsmag_eucl}, the Rips magnitudes of Euclidean cycles are given by:
\[
|t C^{\mathrm{eucl}}_n|_{\mathrm{Rips}}=\sum_{%\substack{
\text{odd }r|n%\\r\neq n}
}\frac nr (e^{-\delta_r t}-e^{-\delta_{r,n}t})+e^{-\delta_n t},
\]
where
\[
\delta_r=\diam(C^{\mathrm{eucl}}_r)=2\sin\left(\pi\frac{\lfloor\frac r2\rfloor}r\right)\quad\text{and}\quad\delta_{r,n}=2\sin\left(\pi\left(\frac1n+\frac{\lfloor\frac r2\rfloor}r\right)\right).
\]
For an odd prime $p$, each odd divisor of $mp$ is of the form $r$ or $rp$ (or both), where $r$ is an odd divisor of $m$. Therefore, assuming $p$ is large enough, so that $p\nmid m$, we can split the Rips magnitude into three summands:
\[
|t C^{\mathrm{eucl}}_{m p}|_{\mathrm{Rips}}=\sum_{\text{odd }r|m}\frac{mp}r (e^{-\delta_r t}-e^{-\delta_{r,mp}t})+\sum_{%\substack{
\text{odd }r|m%\\r\neq m}
}\frac mr (e^{-\delta_{rp} t}-e^{-\delta_{rp,mp}t})+e^{-\delta_{mp} t}.
\]
We can calculate the limit as $p\to\infty$ of each summand individually. To treat the first summand, define a function $\phi\colon \R\to\R$ by the formula
\[
\phi(x)=2\sin\left(\pi\left(x+\frac{r-1}{2r}\right)\right),
\]
with derivative
\[
\phi'(x)=2\pi\cos\left(\pi\left(x+\frac{r-1}{2r}\right)\right).
\]
Observe that $\delta_r=\phi(0)$ and $\delta_{r,mp}=\phi(\frac1{mp})$. Therefore,
\begin{align*}
\lim_{p\to\infty}\frac{mp}r (e^{-\delta_r t}-e^{-\delta_{r,mp}t})&=-\frac1r\lim_{p\to\infty}\frac{e^{-\phi(\frac1{mp})t}-e^{-\phi(0)t}}{\frac1{mp}}\\&=-\frac1r\frac{\mathrm d}{\mathrm dx}{\Big\vert}_{x=0}e^{-\phi(x)t}\\&=\frac1re^{-\phi(0)t}\phi'(0)t\\&=\frac{2\pi t}re^{-2t\sin\left(\pi\frac{r-1}{2r}\right)}\cos\left(\pi\frac{r-1}{2r}\right)\\&=\frac{2\pi t}re^{-2t\cos\left(\frac{\pi}{2r}\right)}\sin\left(\frac{\pi}{2r}\right).
\end{align*}
This takes care of the first summand. The second summand vanishes in the limit, because
\[
\lim_{p\to\infty}e^{-\delta_{rp} t}=\lim_{p\to\infty}e^{-\delta_{rp,mp} t}=e^{-2t}.
\]
Finally, the limit of the third summand is
\[
\lim_{p\to\infty}e^{-\delta_{mp} t}=e^{-2t}.
\]
\end{proof}

\begin{proof}[Proof of Theorem~\ref{theorem-lim-sup-inf}]
Define $\phi_{n,r}(t)=\frac nr(e^{-\delta_r t}-e^{-\delta_{r,n}t})$ for $n\geq r$ and $0$ otherwise. Here $r$ will always be assumed to be odd. We can show that for $t>0$,
\[
\phi_{n,r}(t)\geq 0\qquad\text{and}\qquad\phi_{n+1,r}(t)\geq\phi_{n,r}(t)
\]
always hold. The first of these inequalities reduces to the fact that $\delta_r\leq\delta_{r,n}$. For the second inequality, define $\psi(t)=1+ne^{(\delta_r-\delta_{r,n})t}-(n+1)e^{(\delta_r-\delta_{r,n+1})t}$ and observe that the inequality reduces to $\psi(t)\geq 0$ for $t>0$. To show that this is in fact true, one can then verify that:
\begin{itemize}
\item $\psi(0)=0$,
\item $\lim_{t\to\infty}\psi(t)>0$,
\item $\psi'(0)=(n+1)\delta_{r,n+1}-n\delta_{r,n}-\delta_r>0$ and
\item $\psi'(t)=0$ for at most one $t\in(0,\infty)$.
\end{itemize}
To prove the third bullet point, we can write out the expression explicitly and use the addition theorem for the sine function:
\begin{multline*}
2(n+1)\sin\left(\pi\left(\frac1{n+1}+\frac{r-1}{2r}\right)\right)-2n\sin\left(\pi\left(\frac1n+\frac{r-1}{2r}\right)\right)-2\sin\left(\pi\frac{r-1}{2r}\right)\\
=2\left[(n+1)\sin\left(\frac{\pi}{n+1}\right)-n\sin\left(\frac{\pi}{n}\right)\right]\cos\left(\pi\frac{r-1}{2r}\right)\\+2\left[(n+1)\cos\left(\frac{\pi}{n+1}\right)-n\cos\left(\frac{\pi}{n}\right)-1\right]\sin\left(\pi\frac{r-1}{2r}\right).
\end{multline*}
One can now show that the expressions in square brackets are positive and conclude that the whole expression is positive. This proves the second inequality.

Using the explicit formula for the Rips magnitude, we now have:
\[
|t C^{\mathrm{eucl}}_n|_{\mathrm{Rips}}=\sum_{%\substack{
\text{odd }r|n%\\r\neq n}
}\frac nr (e^{-\delta_r t}-e^{-\delta_{r,n}t})+e^{-\delta_n t}\leq\sum_{\text{odd }r\leq n}\frac nr (e^{-\delta_r t}-e^{-\delta_{r,n}t})+e^{-\delta_n t}.
\]
Therefore,
\begin{align*}
\limsup_{n\to\infty}|t C^{\mathrm{eucl}}_n|_{\mathrm{Rips}}&\leq\lim_{n\to\infty}\left(\sum_{\text{odd }r}\phi_{n,r}(t)+e^{-\delta_nt}\right)\\&=\sum_{\text{odd }r}\lim_{n\to\infty}\phi_{n,r}(t)+\lim_{n\to\infty}e^{-\delta_nt}\\&=\sum_{\text{odd }r}\lim_{n\to\infty}\frac nr (e^{-\delta_r t}-e^{-\delta_{r,n}t})+e^{-2t}\\&=\sum_{\text{odd }r}\frac {2\pi t}r e^{-2t\cos\left(\frac\pi{2r}\right)}\sin\left(\frac\pi{2r}\right)+e^{-2t},
\end{align*}
where the limit is calculated in the same way as in the proof of the previous theorem. The interchange of sum and limit is justified by the Lebesgue monotone convergence theorem. To prove the reverse inequality, let $N$ be an arbitrary positive integer. Let $m=N!$ and note that $m$ is divisible by every odd $r\leq N$. Therefore, by the previous theorem,
\begin{align*}
e^{-2t}+2\pi t\sum_{\text{odd }r\leq N}\frac1r e^{-2t\cos\left(\frac\pi{2r}\right)}\sin\left(\frac\pi{2r}\right)&\leq e^{-2t}+2\pi t\sum_{\text{odd }r|m}\frac1r e^{-2t\cos\left(\frac\pi{2r}\right)}\sin\left(\frac\pi{2r}\right)\\&=\lim_{\substack{p\to\infty\\p\text{ prime}}}|t C^{\mathrm{eucl}}_{mp}|_{\mathrm{Rips}}\\&\leq \limsup_{n\to\infty}|t C^{\mathrm{eucl}}_n|_{\mathrm{Rips}}.
\end{align*}
Taking the limit as $N\to\infty$ establishes the lower bound.

To prove the statement about the lower limit, again start from the explicit formula and note that $r=1$ is a proper odd divisor of any integer $n>1$:
\[
|t C^{\mathrm{eucl}}_n|_{\mathrm{Rips}}=\sum_{%\substack{
\text{odd }r|n%\\r\neq n}
}\frac nr (e^{-\delta_r t}-e^{-\delta_{r,n}t})+e^{-\delta_n t}\geq n (e^{-\delta_1 t}-e^{-\delta_{1,n}t})+e^{-\delta_n t}.
\]
Therefore
\[
\liminf_{n\to\infty}|t C^{\mathrm{eucl}}_n|_{\mathrm{Rips}}\geq\lim_{n\to\infty}\left(n (e^{-\delta_1 t}-e^{-\delta_{1,n}t})+e^{-\delta_n t}\right)=2\pi t+e^{-2t}.
\]
To establish the other inequality, note that by the previous theorem, the prime indices yield a subsequence converging to the lower bound.
\end{proof}

\begin{proof}[Proof of Theorem~\ref{theorem-sup}]
The upper limit of any sequence is a lower bound for its supremum, so
\[
\limsup_{n\to\infty}|t C^{\mathrm{eucl}}_n|_{\mathrm{Rips}}\leq\sup_{n\in\N}|t C^{\mathrm{eucl}}_n|_{\mathrm{Rips}}.
\]
To show the converse inequality, we first observe that
\[
|t C^{\mathrm{eucl}}_n|_{\mathrm{Rips}}\leq|t C^{\mathrm{eucl}}_{2n}|_{\mathrm{Rips}}
\]
holds for any odd $n\in\N$. To see this, we look at their difference. Using the facts that $n$ and $2n$ have the same odd divisors, that $\delta_n=\delta_{n,n}$ and $\delta_{n,2n}=\delta_{2n}=2$, this can be simplified to:
\[
|t C^{\mathrm{eucl}}_{2n}|_{\mathrm{Rips}}-|t C^{\mathrm{eucl}}_n|_{\mathrm{Rips}}=\sum_{\substack{\text{odd }r|n\\r\neq n}}\left(\phi_{2n,r}(t)-\phi_{n,r}(t)\right)+e^{-\delta_n t}-e^{-2 t},
\]
where $\phi_{n,r}$ is as defined in the proof of Theorem \ref{theorem-lim-sup-inf}, where we also showed that $\phi_{n+1,r}\geq\phi_{n,r}$ for all $n$ and $r$. Using this latter fact and the fact that $\delta_n\leq2$, we now have
\[
\phi_{2n,r}(t)-\phi_{n,r}(t)\geq0\qquad\text{and}\qquad e^{-\delta_n t}-e^{-2 t}\geq 0,
\]
so the difference is indeed nonnegative.

Therefore, the supremum may be calculated over the even numbers:
\[
\sup_{n\in\N}|t C^{\mathrm{eucl}}_n|_{\mathrm{Rips}}=\sup_{n\text{ even}}|t C^{\mathrm{eucl}}_n|_{\mathrm{Rips}}.
\]
Now recall that for even $n$ we have:
\[
|t C^{\mathrm{eucl}}_n|_{\mathrm{Rips}}\leq\sum_{\text{odd }r\leq n}\frac nr (e^{-\delta_r t}-e^{-\delta_{r,n}t})+e^{-2t}.
\]
Furthermore, we saw in the proof of Theorem \ref{theorem-lim-sup-inf} that the expression on the right hand side is increasing in $n$ (because $\phi_{n+1,r}\geq\phi_{n,r}$). Taking the supremum of both sides over all even integers $n$, we therefore have:
\[
\sup_{n\text{ even}}|t C^{\mathrm{eucl}}_n|_{\mathrm{Rips}}\leq\lim_{\substack{n\to\infty\\n\text{ even}}}\sum_{\text{odd }r\leq n}\frac nr (e^{-\delta_r t}-e^{-\delta_{r,n}t})+e^{-2t}\leq\limsup_{n\to\infty}|t C^{\mathrm{eucl}}_n|_{\mathrm{Rips}},
\]
where the second inequality is immediate from the proof of Theorem \ref{theorem-lim-sup-inf}.
\end{proof}

\begin{proof}[Sketch proof of Theorem~\ref{theorem-GC}]
By Corollary \ref{ripsmag_geo}, the explicit formulas for Rips magnitudes of geodesic cycles are
\[
|t C^{\mathrm{geo}}_n|_{\mathrm{Rips}}=\sum_{\substack{\text{odd }r|n\\r\neq n}}\frac nr (e^{-\eta_r t}-e^{-\eta_{r,n}t})+e^{-\eta_n t},
\]
where
\[
\eta_r=2\pi\frac{\lfloor\frac r2\rfloor}r\qquad\text{and}\qquad\eta_{r,n}=2\pi\left(\frac1n+\frac{\lfloor\frac r2\rfloor}r\right).
\]
A similar procedure as in the Euclidean case now allows for the calculation of limits of various subsequences, as well as the upper and the lower limit. The main difference is that in this case, the series obtained as the upper limit does not converge anymore:
\[
\limsup_{n\to\infty}|t C^{\mathrm{geo}}_n|_{\mathrm{Rips}}=\sum_{\text{odd }r}\frac{2\pi t}re^{-\pi\frac{r-1}{r}t}+e^{-\pi t}=\infty.
\qedhere
\]
\end{proof}

\bibliographystyle{plain}
\bibliography{rips-magnitude}

\end{document}